\documentclass [a4paper,12pt]{article}
\usepackage[dvips]{graphicx}
\usepackage[english]{babel}
\usepackage{amssymb,latexsym}
\usepackage{mathrsfs}
\usepackage[all]{xy}
\usepackage{color}

\usepackage{epsfig,afterpage}
\usepackage{psfrag}
\usepackage{graphicx}

\usepackage[ansinew]{inputenc}
\input epsf

\usepackage{amsmath}
\usepackage{amssymb}
\usepackage{amsfonts}
\usepackage{multirow}

\newtheorem{theorem}{Theorem}
\newtheorem{lemma}{Lemma}
\newtheorem{remark}{Remark}
\newtheorem{prop}{Proposition}

\newtheorem{defin}{Definition}

\numberwithin{equation}{section}  % this is to number equation with the section they belong
\numberwithin{theorem}{section}
\numberwithin{lemma}{section}
\numberwithin{remark}{section}
\numberwithin{prop}{section}
\numberwithin{assum}{section}
\numberwithin{coro}{section}
\numberwithin{defin}{section}

\def\{{\protect\lbrace}
\def\}{\protect\rbrace}

\def\Real{\operatorname{Re}}
\def\dim{\operatorname{dim}}

\def\sign{\operatorname{sgn}}
\def\ord{\operatorname{ord}}
\def\const{{\rm{const}}}

\newcommand{\pa}{\partial}

\def\alf{\alpha}
\def\bet{\beta}
\def\Del{\Delta}
\def\eps{\varepsilon}
\def\gam{\gamma}
\def\Gam{\Gamma}
\def\Gamo{\Gam_0}
\def\phi{\varphi}
\def\om{\omega}

\def\la{\lambda}

\def\ti{\widetilde}

\newcommand{\bR}{\mathbb{R}}
\newcommand{\bC}{\mathbb{C}}
\newcommand{\bZ}{\mathbb{Z}}
\newcommand{\bN}{\mathbb{N}}

\newcommand{\FF}{\mathscr {F}}
\newcommand{\D}{\mathscr {D}}
\newcommand{\MM}{{\mathfrak M}}

\newcommand{\gamap}{\gam_{\alf}^+}
\newcommand{\gamam}{\gam_{\alf}^-}
\newcommand{\gamapm}{\gam_{\alf}^{\pm}}

\newcommand{\gamabp}{\gam_{\alf, \bet}^+}

\def\Del{\Delta}

\newcommand{\M}{S}

\def\Ci{{C^{\infty}}}
\def\Com{{C^{\om}}}

\let\<\langle
\let\>\rangle

\newenvironment{proof}
 {\begin{trivlist} \item[\hskip \labelsep {\bf Proof}\hspace*{3 mm}]}
 {\hfill$\Box$\end{trivlist}}

\begin{document}

\title{Singularities of the geodesic flow on surfaces with pseudo-Riemannian metrics}

\author{A.O. Remizov\footnote{Supported by the FAPESP visiting professor grant 2014/04944-6.} \, and F.
Tari\footnote{Partially supported by the grants  FAPESP 2014/00304-2, CNPq 301589/2012-7,  472796/2013-5.}}

\maketitle

\begin{abstract}
We consider a pseudo-Riemannian metric that changes signature along a smooth curve on a surface, 
called the discriminant curve. The discriminant curve separates the surface locally into a Riemannian 
and a Lorentzian domain. We study the local behaviour and properties of geodesics  
at a point on the discriminant where the isotropic direction is tangent to the discriminant curve. 
\end{abstract}

\renewcommand{\thefootnote}{\fnsymbol{footnote}}
\footnote[0]{2010 Mathematics Subject classification 53C22, 53B30, 34C05.} 

\section{Introduction}
The work in this paper is a part of an ongoing
research on understanding the geometry of surfaces
endowed with a signature varying metric (see, for example,
\cite{GKT, GR, contours, degpairs, KT, Kossowski, KosKri, Mier, PR11, Pelletier,
Rem-Pseudo, Rem15, AmaniF, Steller,T, CaratheodoryR31}).
We consider here the behavior of geodesics on a surface $S$ endowed with a pseudo-Riemannian metric
given in local coordinates by
\begin{equation}
ds^2 = a(x,y) \, dx^2 + 2b(x,y) \, dx dy + c(x,y) \, dy^2
\label{1}
\end{equation}
where the coefficients $a,b,c$ are smooth (that means $\Ci$ unless stated otherwise)
functions on an open set $U \subset \mathbb R^2$.

We assume in all the paper that
the discriminant function $\Del(x,y) = (ac-b^2)(x,y)$
vanishes on a regular curve $\D$,
which is called {\it signature changing curve} or simply {\it discriminant curve} of the metric \eqref{1}.
The discriminant curve $\D$ separates (at least, locally) the surface $\M$
into a Riemannian ($\Del>0$) and a Lorenzian ($\Del<0$) domain.

At any point in the Lorentzian domain, there are two {\it lightlike} or {\it isotropic} directions
that consist of vectors with zero length and there is
one double isotropic direction at any point on $\D$.
The {\it isotropic curves} are integral curves of the equation
\begin{equation}
a(x,y) \, dx^2 + 2b(x,y) \, dx dy + c(x,y) \, dy^2 = 0.
\label{2}
\end{equation}

Isotropic curves are geodesics (except in the case below) in the metric \eqref{1}
when defined, for instance, as extremals of the action functional (the arc-length parametrization
is not defined for these curves), see \cite{Rem15}.
A non-isotropic geodesic is called {\it timelike} ({\it spacelike}) if $ds^2>0$ ($ds^2<0$)
along the geodesic.
The exception is when the isotropic curve coincides with the discriminant curve $\D$
and is a {\it singular solution} of equation \eqref{2}, i.e.,
it is the envelop of one-parameter family of isotropic curves.
This is the case $Z$ (Table \ref{Tab1}) and the reason why the singular solution
is not a geodesic is explained in Appendix~B.

When the unique isotropic direction is transverse to $\D$, the isotropic curves form a family of cusps
(configuration $C$ in Figure \ref{pic1}).
At points where the isotropic direction is tangent to $\D$,
the isotropic curves have generically one of the configurations $D_s,D_n,D_f$
in Figure \ref{pic1}.
Finally, the configuration $Z$ in Figure \ref{pic1}
occurs in the case when the isotropic direction is tangent to $\D$ at all points on $\D$.
(See Proposition~\ref{Pro2} for more details.)

\begin{figure}[ht]
\begin{center}
\includegraphics[height=4.5cm]{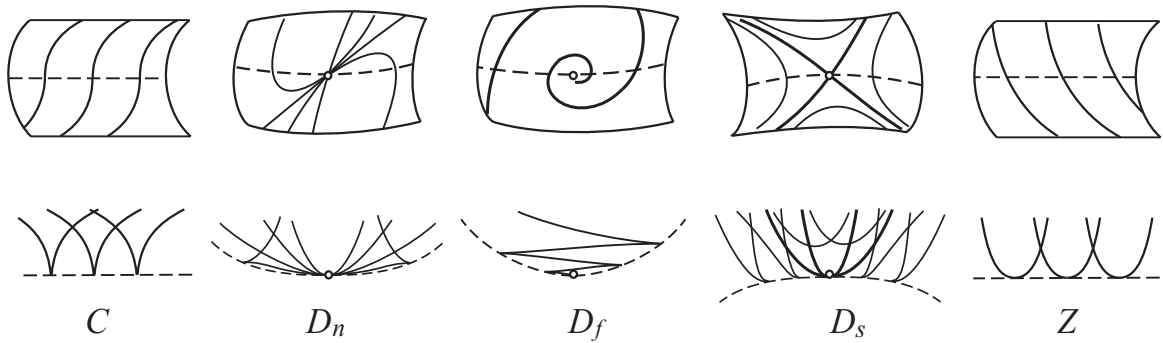}
\end{center}
\caption{
Top figures represent the integral curves of the field \eqref{4} on the isotropic surface $\FF$.
The bottom figures represent their projections to $S$, i.e., isotropic geodesics.
The dashed lines represent the criminant (in the top figures) and the discriminant curve (in the bottom figures).
}
\label{pic1}
\end{figure}

The geodesic flow generated by the metric \eqref{1} has singularities at every point $q \in \D$.
The geodesics cannot pass through $q \in \D$ in an arbitrary tangent direction
but only in certain directions said to be {\it admissible}, and the isotropic direction is always one of them.
Singularities of geodesic flow at generic points $q \in \D$ (excluding isolated points)
are studied in \cite{GR, Rem-Pseudo, Rem15}.
The excluded isolated points are the following:
%(a) points where the isotropic direction is tangent to $\D$,
%(b) points of intersection of $\D$ with the closure of the parabolic set of $\M$.
\begin{itemize}
\item[{\rm (a)}] points where the isotropic direction is tangent to $\D$,
\item[{\rm (b)}] points of intersection of $\D$ with the closure of the parabolic set of $\M$.
\end{itemize}

In this paper, we complete the study in \cite{GR, Rem-Pseudo, Rem15} for the case~(a).
We give in Theorems~\ref{T4} and \ref{T5} the configurations of the geodesic in case (a)
that correspond to the configurations $D_s$ and $D_n, D_f$ of the isotropic geodesics, respectively.
We start our study with the configuration $Z$ (Theorem~\ref{T3}), which is not generic.
However, it provides a graphical illustration of some properties common with the cases $D_s, D_n, D_f$
(for instance, there are no geodesics outgoing from the point of interest to the Lorenzian domain).

In Appendix~A we recall some results on local normal forms of resonant and non-resonant
vector fields that we use in the paper.
Finally, in Appendix~B we briefly consider naturally parametrized geodesics defined as extremals of
the action functional.

% % % % % % % % % % % % % % % % % % % % % % % % % % % % % % % % % % % % % % % % % %

\section{Isotropic curves and geodesics}\label{General}

As the discriminant curve $\D$ is supposed to be a regular curve, the coefficients $a,b,c$ of the metric
\eqref{1} cannot vanish simultaneously at any point in a neighborhood of $\D$.
We assume, without loss of generality, that near the point of interest the coefficient $c(x,y)>0$
and write equation \eqref{2} as the implicit differential equation
\begin{equation}
F(x,y,p) := c(x,y) p^2 + 2b(x,y) p + a(x,y) = 0,
\label{3}
\end{equation}
where $p=dy/dx$ is the non-homogeneous coordinate in the projective tangent bundle $PT\M$.
At any point $q=(x,y)$ in the Lorenzian domain ($\Del < 0$), equation \eqref{3} has two simple real roots
which correspond to the pair of transversal isotropic directions at $q$.
At any $q \in \D$, equation \eqref{3} has a double root $p_0(q)=-\frac{b}{c}(q)$, which corresponds to the unique
isotropic direction at $q$. Generically, the direction $p_0$ is transverse to $\D$ at almost all points
while tangency (of the first order) can occur at isolated points on $\D$ only.

In the $(x,y,p)$-space, equation \eqref{3} defines a surface $\FF$ called
the {\it isotropic surface} of the metric \eqref{1}.
On $\FF$ is defined a direction field
which is the intersection of the contact planes $dy = pdx$ with the tangent planes of $\FF$.
This direction field is parallel to the vector field
\begin{equation}
\frac{1}{2}F_p \pa_x + \frac{1}{2}pF_p\pa_y - \frac{1}{2}(F_x+pF_y) \pa_p.
\label{4}
\end{equation}
Formula \eqref{4} differs from what is generally used by the factor $\frac{1}{2}$. This is for convenience
in the computations carried out in this paper.

Although the vector field \eqref{4} is defined in the $(x,y,p)$-space, we will consider it
only on the isotropic surface $\FF$ which is an invariant surface of \eqref{4}.
The isotropic curves are the projections
of the integral curves of \eqref{4} on $\FF$ to the $(x,y)$-plane along the $p$-direction,
called {\it vertical}.
The set of points on the surface $\FF$ where the projection $\pi \colon \FF  \to \bR^2$
is singular (i.e., where $F=F_p=0$) is called the {\it criminant} of equation \eqref{3}.
That is, the criminant consists of the points $(q,p_0(q))$ with $q \in \D$.
The projection of the criminant on the $(x,y)$-plane is the discriminant curve $\D$.

\begin{remark}
{\rm
In a neighborhood of any point $q_* \in \D$, there exist local coordinates where the metric
is diagonal (i.e., $b\equiv 0$) and the discriminant curve $\D$ is given by the equation $a(x,y)=0$.
Then $p_0(q) \equiv 0$ at all $q \in \D$, the derivatives $a_x, a_y$ do not vanish simultaneously
(since $\Del$ is a regular function) and $c(q_*) \neq 0$.
We assume, without loss of generality, that $c(q_*) > 0$.
}
\label{Rem1}
\end{remark}

As in Riemannian geometry, non-parametrized geodesics in pseudo-Riemannian metrics
can be defined as extremals of the length functional or as extremals of the action
functional after ignoring the natural parametrization, see \cite{Rem15}.
In the fist case, a special problem arises.
The Lagrangian of the length functional has the form $L = \sqrt{F}$,
where $F$ defined in \eqref{3} vanishes on the isotropic surface $\FF$,
and the corresponding Euler--Lagrange equation is not defined on $\FF$.
Also, the arc-length parametrization is not defined for isotropic curves.
However, such a problem does not arise if we define geodesics as extremals of the action functional.

The standard projectivization $T\M \to PT\M$,
which means forgetting the natural parametrization of geodesics)
transforms the field \eqref{4}
into a direction field on  $PT\M$ parallel to
\begin{equation}
2\Del \pa_x + 2p\Del \pa_y + M \pa_p,
\label{5}
\end{equation}
where $M$ is a cubic polynomial in $p$ given by
$
M(q,p) = \sum\limits_{i=0}^3 \mu_i(q)p^i
$
with the coefficients
\begin{equation*}
\begin{aligned}
&\mu_0 = a(a_y-2b_x) +a_xb,\\
&\mu_1 = b(3a_y-2b_x) + a_xc - 2ac_x, \\
&\mu_2 = b(2b_y-3c_x) + 2a_yc - ac_y, \\
&\mu_3 = c(2b_y-c_x) - bc_y.
\end{aligned}
\end{equation*}
See \cite{Rem15} for more details.
By Theorem 1 in \cite{GR}, the isotropic surface $\FF$ is an invariant surface of the field \eqref{5},
and integral curves of the 2-distribution $dy = pdx$ lying on $\FF$ are integral curves of \eqref{5}.
Hence all isotropic curves (except for those that coincide with $\D$)
with an appropriate choice of parametrization are extremals of the action functional.
Consequently, they are geodesics.

Consider the field \eqref{5} at a point $(q,p)$, $q \in \D$.
Since $\Del (q)=0$, the vertical line passing through $q$ is an integral curve of \eqref{5}. If $M(q, p)\neq 0$,
it is the unique integral curve  of \eqref{5} passing through  $(q, p)$. The projection of this curve to the $(x,y)$-plane along the $p$-direction is the point $q$, so is not a geodesic.
Hence geodesics can pass through a point $q \in \D$
only at {\it admissible} tangential directions $p$ given by  $M(q, p)= 0$. One can show (see \cite{Rem-Pseudo}) that
\begin{equation}
M(q,p) = \frac{1}{3}(p-p_0) ( 2(\Del_x+p\Del_y) + M_p), \ \ \ \forall \, q \in \D.
\label{7}
\end{equation}

From \eqref{7}, it follows that the isotropic direction $p_0$ is always a root of the cubic polynomial $M(q,p)$.
Furthermore, if $p$ is a double root of $M(q,p)$, then the direction $p$ is tangent to the curve $\D$ at $q$.
Indeed, if $p \neq p_0$, then substituting $M = M_p = 0$ in \eqref{7}, we get $\Del_x+p\Del_y=0$.
If $p = p_0$, then differentiation \eqref{7} with respect to $p$ and substitution $M=M_p=0$ yields $\Del_x+p_0\Del_y=0$.

\medskip

Let $K$ denote the Gaussian curvature of $\M \setminus \D$.
(Generically $K(q) \to \infty$ as $q \to \D$.)
The function $K_1=\Del^2 K$ is well defined and is smooth on the whole surface $\M$
and is an isometric invariant.
The set $K_1=0$ is the closure of the parabolic set of $\M$.

Before giving some properties of $M$, we introduce the following notation:
$p_0 \pitchfork \D$ means that the direction $p_0$ is transverse to the curve $\D$,
$p_0 \| \D$ and $\ord \, (p_0 \| \D) =1$ mean respectively that $p_0$ is tangent to $\D$
and has first order tangency with $\D$.
Finally, $(\eps_1,\eps_2)$ and $(\la_1,\la_2,0)$ denote the spectrums of the linear parts
of the vector fields \eqref{4} and \eqref{5} at the point $(q,p_0)$, $q \in \D$, respectively.
%(We emphasize that the field \eqref{4} is considered on the isotropic surface $\FF$ only.)

\begin{prop}
Let $q \in \D$, then the following statements hold.
\begin{itemize}
\item[{\rm (i)}]
The cubic polynomial $M(q,p)$ has a unique real root $p_0$ if and only if $K_1(q)<0$.
\item[{\rm (ii)}]
$p_0$ is a simple root of  $M(q,p)$  if and only if $p_0 \pitchfork \D$ at $q$,
or equivalently,
the $p$-component of the field \eqref{4} does not vanish at $(q,p_0)$.
\item[{\rm (iii)}]
$M(q,p)$ has three distinct real roots $p_0, p_1, p_2$  if and only if $K_1(q) > 0$ and $p_0 \pitchfork \D$ at $q$.
\item[{\rm (iv)}]
$M(q,p)$ has a double root $p_1=p_2\ne p_0 $ if and only if $K_1(q) = 0$ and $p_0 \pitchfork \D$ at $q$.
The double non-isotropic direction $p_{1}$ is tangent to $\D$ at $q$.
\item[{\rm (v)}]
$M(q,p)$ has a double root  $p_0 = p_1 \neq p_2$ if and only if $K_1(q) > 0$ and $p_0 \| \D$ at $q$.
\item[{\rm (vi)}]
The isotropic direction $p_0 \| \D$ at $q$ if and only if \eqref{4} vanishes at $(q,p_0)$.
Then $\eps_1+\eps_2 \neq 0$. Moreover, $\ord \, (p_0 \| \D) =1$ if and only if  $\eps_{1}, \eps_{2}\neq 0$, that is, $(q,p_0)$ is a node, saddle or focus of \eqref{4}.
\item[{\rm (vii)}]
If $p_0 \| \D$ at all points on $\D$, then the $p$-component of the field \eqref{4} vanishes on the criminant. Hence the criminant consists of singular points of \eqref{4} with one of the eigenvalues $\eps_{1},\eps_{2}$ equal to zero and the other is distinct from zero.
\item[{\rm (viii)}]
If $(q,p)$ is a singular point of the field \eqref{5}, then $\la_1=2(\Del_x(q)+p\Del_y(q))$ and $\la_2=M_p(q,p)$.
\end{itemize}
\label{Pro1}
\end{prop}

\begin{proof}
It is convenient to take the point $q$ to be the origin and choose the local coordinates in Remark \ref{Rem1}.
Then \eqref{7} becomes
$$
M(0,p) = -cp Q(p), \ \ \  Q(p) = c_xp^2 -2a_yp -a_x,
$$
the quadric polynomial $Q$ has discriminant $\delta = a_xc_x + a_y^2$.
Using the Brioschi formula \cite{Struik},
we get $K_1(0) = {c(0)\delta}(0)/4$. Since the coefficient $c$ is positive,
it follows that $K_1(0)$ and $\delta(0)$ have the same sign.

For (i), the cubic polynomial $M(0,p)$ has only one real root $p_0$ if and only if
$\delta(0)<0$ or $\delta(0)=0$ and $p_0=0$ is a double root of $Q(p)$.
In the second case we have $a_x(0) = 0$, which together with  $\delta(0)=0$
gives  $a_x(0) = a_y(0) = 0$.
This contradicts the assumption that $\Del=ac$ is regular at the origin.

The proof of (ii) follows by direct calculation.
For (iii), $M(0,p)$ has three distinct real roots  if and only if
$Q(p)$ has two distinct real roots $p_1 \neq p_2$ and $p_0 \neq p_{1},p_{2}$.
Clearly, these conditions are equivalent to $K_1(0)>0$ and $p_0 \pitchfork \D$.
The statement (iv) follows from \eqref{7}.
The proofs of the remaining statements follow by direct calculations and are omitted.
\end{proof}

We have the following about isotropic geodesics.

\begin{prop}
Let $q \in \D$, then the following statements hold.
\begin{itemize}
\item[{\rm $C$.}]
If $p_0 \pitchfork \D$ at $q \in \D$, then the germ of \eqref{3} at $(q,p_0)$
is smoothly equivalent to $p^2 = x$ {\rm (}Cibrario normal form{\rm )}.
\item[{\rm $D$.}]
If $\ord \, (p_0 \| \D) =1$ at $q \in \D$
and between the eigenvalues $\eps_{1}, \eps_{2}$ there are no non-trivial resonances
\begin{equation}
s_1\eps_1 + s_2\eps_2 = \eps_j, \ \  s_{1},s_2  \in \bZ_+, \ \ s_1+s_2 \ge 1, \ \  j=1,2,
\label{8}
\end{equation}
then the germ of \eqref{3} at $(q,p_0)$ is smoothly equivalent to
$p^2 = y-\eps x^2$, where $\eps = \eps_1 \eps_2 \neq 0,\frac{1}{16}$ {\rm (}Dara-Davydov normal form{\rm )}.
The excluded values $\eps = 0,\frac{1}{16}$ correspond the cases when
$\eps_{1}\eps_{2} = 0$ or $\eps_1 = \eps_2$.
\item[{\rm $Z$.}]
If $p_0 \| \D$ at all points on $\D$, then the germ of \eqref{3} at $(q,p_0)$
is smoothly equivalent to $p^2 = y$ {\rm (}Clairaut normal form{\rm )}.
\end{itemize}
\label{Pro2}
\end{prop}

\begin{proof}
The proof of the case $C$ in the smooth category can be found in \cite{A-Geom}
(the earlier proof in the analytic category can be found in \cite{cib}).
The proof of the case $D$ can be found in \cite{Dav85} or \cite{Dav94}.
The proof of the case $Z$ can be found in \cite{Dav-Japan}
(it can be also proved using the approach in \cite{A-Geom} for the case $C$).
\end{proof}

\begin{remark}
{\rm
It is worth observing that Cibrario normal form $p^2 = x$ (case $C$ in Proposition~\ref{Pro2}) 
is valid in the analytic category, while Dara-Davydov normal form $p^2 = y-\eps x^2$ 
(case $D$ in Proposition~\ref{Pro2}) is not; see \cite{Mier}.
}
\label{RemAnal}
\end{remark}

Table~\ref{Tab1} combines the information in Propositions \ref{Pro1} and \ref{Pro2} and gives
a classification of singularities of the geodesic flow generated by the metric \eqref{1}.
The first column contains  the names of the singularities.
The singularities  $C_i$, $i=1,2,3$ and $D_s, D_n, D_f$ are topologically stable,
while the singularity $Z$ is not and has infinite codimension.
Moreover, the topological equivalence eliminates the modulus parameter $\eps$
in the normal form $p^2 = y-\eps x^2$ in each of three cases $D_s, D_n, D_f$ (see \cite{Dav85, Dav94}).
The second and third columns are the conditions for the occurrence of the singularities. The last column
gives the normal forms of the equation of the isotropic geodesics \eqref{3},
see also Figure~\ref{pic1}.

\begin{table}[htb]
\begin{center}
\caption{
Classification of singularities of the geodesic flow.
%in pseudo-Riemannian metrics.
}
{\footnotesize
\begin{tabular}{|c|c|c|c|c|}
\hline
case &  condition 1 & condition 2 & real roots of $M(q,p)$ & normal form of \eqref{3} \\
\hline
$C_1$ & $K_1(q)<0$ &  $p_0 \pitchfork \D$ at $q$ & $p_0$ (simple) &
$p^2 = x$  \\
\hline
$C_2$ & $K_1(q)=0$ &  $p_0 \pitchfork \D$ at $q$ & $p_0 \neq p_1 = p_2$ &
$p^2 = x$ \\
\hline
$C_3$ & $K_1(q)>0$ & $p_0 \pitchfork \D$ at $q$ & $p_0, p_1, p_2$ (all simple) &
$p^2 = x$ \\
\hline
$D_s$ &  $K_1(q)>0$ & $\ord \, (p_0 \| \D) =1$ at $q$ &
$p_0 = p_1 \neq p_2$ & $p^2 = y-\eps x^2$  \\
%---
{} & & \eqref{4} is a saddle at $(q,p_0)$ on $\FF$ & {} & $\eps < 0$  \\
\hline
%%%%%%%%%%%%%%%%%%%%%
$D_n$ &  $K_1(q)>0$ & $\ord \, (p_0 \| \D) =1$ at $q$ &
$p_0 = p_1 \neq p_2$ & $p^2 = y-\eps x^2$ \\
{} & & \eqref{4} has node at $(q,p_0)$ on $\FF$& {} & $0 < \eps < \frac{1}{16}$    \\
\hline
$D_f$ &  $K_1(q)>0$ & $\ord \, (p_0 \| \D) =1$ at $q$ &
$p_0 = p_1 \neq p_2$ & $p^2 = y-\eps x^2$  \\
%---
{} & & \eqref{4} is a focus at $(q,p_0)$ on $\FF$& {} & $\eps > \frac{1}{16}$   \\
\hline
$Z$ & $K_1(q)>0$ &  $p_0 \| \D$  at $\forall q \in \D$ & $p_0 = p_1 \neq p_2$  $\forall q \in \D$  &
$p^2 = y$   \\
\hline
\end{tabular}
\label{Tab1}
}
\end{center}
\end{table}

The following results about the behavior of the geodesics
outgoing from a point $q \in \D$ with the tangential direction $p$ being a
simple root of $M(q,p)$ are established in \cite{Rem-Pseudo}.
We start with geodesics passing through a point on the discriminant curve $\D$ with non-isotropic admissible directions. It covers all the cases in Table \ref{Tab1} except $C_2$.

\begin{theorem}[\cite{Rem-Pseudo}]
\label{T1}
Let $p_i$, $i=1$ or $2$, be a non-isotropic simple real root of the cubic polynomial $M(q,p)$ at $q \in \D$.
Then to the admissible direction $p_i$ corresponds a unique smooth geodesic
passing through the point $q$.
\end{theorem}

The next theorem is proved in \cite{Rem-Pseudo} for the cases $C_1, C_3$.
However, from its proof it follows that it is also valid in the case $C_2$,
since it deals with the root $p_0$ only.
Theorems \ref{T1} and \ref{T2} give complete information about the configuration of geodesics
outgoing from a point $q \in \D$ with all possible admissible directions in the cases $C_1, C_3$.

\begin{theorem}[\cite{Rem-Pseudo}]
\label{T2}
Suppose that the isotropic root $p_0$ of $M(q,p)$ at $q\in \D$ is simple.
Then to $p_0$ corresponds a one-parameter family $\Gamo$ of geodesics
outgoing from the point $q$.
There exist smooth local coordinates centered at $0$ such that
the discriminant curve $\D$ coincides with the $x$-axis,
the isotropic direction $p_0(q)=\infty$
and the geodesics $\gamapm \in \Gamo$ are semi-cubic parabolas
\begin{equation}
x= \alf \tau^3 X_{\alf}(\tau), \quad
y= \tau^2 Y_{\alf}(\tau), \quad
\alf \in \bR^+,
%\quad \tau \in \bR, \ \
%\tau \in (\bR, 0)_{\alf}, \quad
\label{9}
\end{equation}
where $X_{\alf}, Y_{\alf}$ are smooth functions, $X_{\alf}(0)=1$, $Y_{\alf}(0)=\pm 1$.
Here the superscripts $+$ and $-$ distinguish geodesics outgoing from the origin in
the semiplanes $y > 0$ and $y < 0$, respectively.
The family \eqref{9} contains the three types of geodesics:
$\gamap$ are timelike if $\alf>1$, spacelike if $\alf<1$, isotropic if $\alf=1$;
all $\gamam$ are timelike. See \mbox{\rm Figure~\ref{pic2}}.
\end{theorem}

%%%%%%%%%%%%%%%%%%%%%%%%%%%%%%%%%%%%%%%%%
\begin{figure}[ht]
\begin{center}
\includegraphics[height=4.1cm]{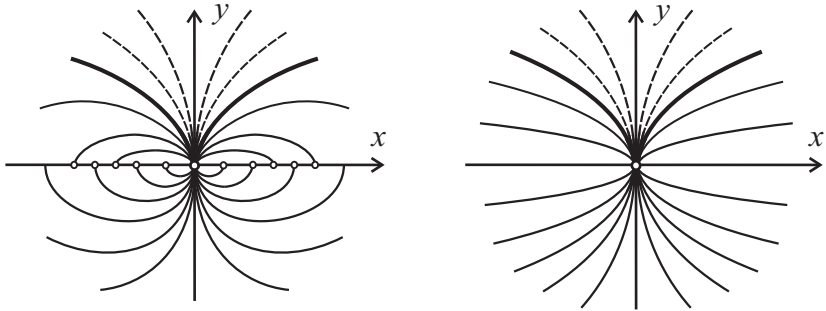}
\end{center}
\caption{
Two examples of configurations of the geodesics outgoing from $q\in \D$ in Theorem~\ref{T2}.
Here the discriminant curve $\D$ coincides with the $x$-axis.
Timelike, spacelike and isotropic geodesics are solid, dashed and bold solid lines, respectively.
}
\label{pic2}
\end{figure}

\begin{remark}
{\rm
1.
For any $\alf \neq 0$ formula \eqref{9} defines
a semi-cubic parabola, while for $\alf = 0$ we have the limiting case of the semi-cubic parabolas:
the two branches are glued together to form the geodesics $\gam^{\pm}_0$
which are the two semi-axes $y>0$ and $y<0$.

2. An interesting problem is to determine whether a geodesic starting from a point $q \in \D$
returns to the discriminant curve $\D$. For instance, two examples of different behavior of geodesics
are presented in Figure \ref{pic2}.
This problem is studied in \cite{Rem15} for metrics possessing differentiable groups of symmetries;
see also \cite{GKT, KT} for related work.
}
\label{Rem2}
\end{remark}

% % % % % % % % % % % % % % % % % % % % % % % % % % % % % % % % % % % % % % % % % %

\section{The main results}

We deal here with the cases  $D_n, D_f, D_s$ and $Z$ in Table \ref{Tab1} and determine the
behavior of the geodesics near a point $q\in \D$, where the cubic polynomial $M(q,p)$
has the double root $p_0=p_1$.

From now on, we will assume that in the cases $D_n, D_s$
between the numbers $\eps_{1}, \eps_{2}, 1$ there are no non-trivial integer relations
\begin{equation}
% \left.
 \begin{aligned}
 s_1\eps_1 + s_2\eps_2 + s_3 &= \eps_j, \\ % \ \, j=1,2,
 s_1\eps_1 + s_2\eps_2 + s_3 &= 1, \\
 \end{aligned}
%  \right\}
\qquad
|s| := \sum_{i=1}^3 s_i \ge 1, \ \ s_i \in \bZ_+.
\label{10}
\end{equation}
Clearly, this condition implies the absence of the resonances \eqref{8} in Proposition~\ref{Pro2}.

Thus in all cases $D_n, D_f, D_s$ and $Z$ there exist smooth local coordinates such that
equation \eqref{3} has the normal form
$p^2 = r(x,y)$, where $r(x,y)=y-\eps x^2$.
The cases $D_n, D_f, D_s$ and $Z$ correspond, respectively, to
$0 < \eps < \frac{1}{16}$, $\eps > \frac{1}{16}$, $\eps < 0$ and $\eps = 0$.
Then the metric \eqref{1} has the form
\begin{equation}
ds^2 = \om(x,y) \, (r(x,y) dx^2 - dy^2) + \Theta, \ \ \om(0,0)=-1, \ \ r(x,y)=y-\eps x^2,
\label{11}
\end{equation}
where $\om$ is a smooth function and $\Theta$ is a smooth metric of the form \eqref{1}
whose coefficients are identically zero in the Lorenzian domain.
Consequently, its coefficients are infinitely flat on the discriminant curve $r(x,y)=0$.
The set of such functions
is the ideal $\MM^{\infty}(r) \subset \<r^n\>$ with any $n \in \bN$
(in the ring of smooth functions).

The assumption $\om(0,0)=-1$ can be achieved by multiplication of the metric by a non-zero constant,
which does not change geodesics.
Then $y>\eps x^2$ (resp. $y<\eps x^2$) is the Lorenzian (resp. Riemannian) domain of the surface.
We shall distinguish geodesics outgoing from the origin to the region $y>\eps x^2$
(resp. $y<\eps x^2$) using the superscripts $+$ (resp. $-$).

After dividing by appropriate factor $(-\om + \ldots)$, the field \eqref{5}
%for the metric \eqref{11} is
is
\begin{equation}
2\om r \pa_x + 2p\om r \pa_y + \bar{M}\pa_p,
\label{12}
\end{equation}
with
$
\bar{M}(x,y,p)= \om_x p^3 + (r\om_y+2\om)p^2 - (r\om_x +2\eps x \om) p - r(r\om_y+\om) + \ldots,
$
where the dots mean terms that belong to the ideal $\MM^{\infty}(r)$.

The singular points of the field \eqref{12} are given by the equations $r(x,y)=0$ and $\bar{M}(x,y,p)=0$.
On the discriminant curve $r(x,y)=0$ and away from the origin, the cubic polynomial
$\bar{M} = p (\om_x p^2 + 2\om p - 2\eps x \om)$
has the simple isotropic root $p_0=0$ and two more distinct real roots $p_{1},p_{2}$
being the roots of the quadratic polynomial $\om_x p^2 + 2\om p - 2\eps x \om$
whose discriminant is strictly positive on $r(x,y)=0$ away from the origin.

Hence in the cases $D_n, D_f, D_s$  at every point on the discriminant curve except the origin
there are three different admissible directions $p_0, p_1, p_2$, while at the origin two of them coincide:
$p_0 = p_1=0$ and $p_2=-2\om /\om_x$.
In the case $Z$ we have $p_0 = p_1=0$ and $p_2=-2\om /\om_x$ at all points on the discriminant curve.
The case $\om_x = 0$ is also included:
then $\bar{M}$ considered as a cubic polynomial over $\bR P^1 = \bR \cup \infty$
has a simple non-isotropic root $p_2=\infty$.

In all cases $D_n, D_f, D_s$ and $Z$, Theorem~\ref{T1} establishes the existence of a unique
smooth geodesic passing through the origin with non-isotropic admissible direction $p_2$, while
Theorem~\ref{T2} is not applicable for the double isotropic direction $p_0 = p_1=0$,
since the spectrum of the linear part of the field \eqref{12} at the origin contains three zero eigenvalues.

To study the behavior of the geodesics passing through the origin
with the double isotropic direction $p_0 = p_1=0$, we consider the blowing up
\begin{equation}
\Psi \colon (x,u,p) \mapsto (x,y,p), \ \ y = \eps x^2 +up^2, \ \ u \in \bR P^1 = \bR \cup \infty.
\label{13}
\end{equation}
The mapping $\Psi$ is one-to-one except on the plane $\Pi = \{(x,u,p) \colon p=0\}$,
whose image is the line $\Psi (\Pi) = \{ (x,y,p) \colon y-\eps x^2 = p = 0 \}$.
The mapping $\Psi$ is a local diffeomorphism at all points except on $\Pi$. It has an inverse defined
on $\mathbb R^3\setminus \Psi (\Pi)$ given  by
$$\Psi^{-1}(x,y,p) = \biggl(x, \frac{r(x,y)}{p^2},p \biggr), \quad r=y-\eps x^2.$$

Observe that there are no geodesics corresponding to integral curves of the field \eqref{12}
coinciding with the curve $\Psi (\Pi)$.
Indeed, if $\Psi (\Pi)$ is an integral curve of \eqref{12}, then
the identities $y \equiv \eps x^2$, $p \equiv 0$ hold.
Form the first of them we have $p \equiv 2\eps x$, which contradicts to the second one if $\eps \neq 0$.
The remaining case $\eps = 0$ is more complicated, and we give the proof of this statement
in Appendix~B using the equation of naturally parametrized geodesic.

Away from $\Psi (\Pi)$, the map $\Psi^{-1}$ sends the isotropic surface $\FF$
which is an invariant surface of the field \eqref{12} to the plane $u=1$.
Hence, the field \eqref{12} corresponds to a smooth field in the $(x,u,p)$-space (away of $\Pi$)
which has $u=1$ as invariant plane.
Taking this into account, we obtain from \eqref{12} and \eqref{13} that
the last field, after dividing by the common factor $p$, is
\begin{equation}
2\om up \pa_x + \bigl(pM_1 - 2\eps x\om + \ldots) \pa_p + (u-1)\bigl(2uN_1 + \ldots) \pa_u,  \\
\label{14}
\end{equation}
with
$$
\begin{array}{l}
M_1(x,u,p) = p(up \om_y + \om_x)(1-u) + \om (2-u), \\
N_1(x,u,p) = up^2 \om_y + p\om_x + \om,
\end{array}
$$
where
$\om, \om_x, \om_y$ are evaluated at $(x,\eps x^2+up^2)$ and
the dots mean terms that belong to the ideal $\MM^{\infty}(up)$.

Since the first and second components of \eqref{14} vanish on the $u$-axis ($x=p=0$),
the $u$-axis is the unique integral curve of the field \eqref{14} passing through any
point satisfying the conditions $x=p=0$ and $u(u-1)N_1(x,u,p) \neq 0$.
From $y = \eps x^2 +up^2$ it follows that the map $\pi \circ \Psi (x,u,p) = (x,y)$
sends the $u$-axis in the $(x,u,p)$-space to the origin in the $(x,y)$-plane.
Consequently, the $u$-axis does not correspond to a geodesic.
This leads to the necessary condition $u(u-1)N_1(0,u,0) = 0$.
Since $N_1(0,u,0) = \om(0,0) \neq 0$, we get only two admissible values: $u=0$ and $u=1$.

\begin{lemma}
To the admissible value $u=0$ does not correspond any geodesic passing through the origin in the $(x,y)$-plane
and tangent to the isotropic direction.
\label{Lem31}
\end{lemma}

\begin{proof}
The germ of the field \eqref{14} at the origin in the $(x,u,p)$-space has the form \eqref{48} with $n=3$
(Appendix~A).
Indeed, its components belong to the ideal $I$ generated by the components
$\bigl(pM_1 - 2\eps x\om + \ldots)$ and $(u-1)\bigl(2uN_1+ \ldots)$, and
the center manifold $W^c$, given by the equations $pM_1 - 2\eps x\om = 0$ and $u=0$,
consists of the  singular points of the field \eqref{14}.
The spectrum of the linear part of \eqref{14} at any singular point is $(\la_1,\la_2,0)$
with $\la_{1}=2\om$ and $\la_2 = -2\om$,
so the resonance $\la_1+\la_2=0$ holds.
It follows by Theorem \ref{PT4}  (Appendix~A) that the germ \eqref{14} is topologically equivalent to
$\xi \pa_{\xi} - \eta \pa_{\eta}$.
By Theorem \ref{PT6}  (Appendix~A), it is smoothly orbitally equivalent
to the normal form \eqref{4.15}.
Moreover, if the 2-jet of \eqref{14} is generic,
it is smoothly orbitally equivalent to \eqref{4.16}, that is,
$\xi \pa_{\xi} - \eta \pa_{\eta} +  \xi\eta \pa_{\zeta}$.

Any of these normal forms shows that there are only two integral curves passing through the origin, which
coincide with the $\xi$-axis and $\eta$-axis in the normal forms coordinates.
On the other hand, it is easy to see that the $p$-axis and $u$-axis are integral curves of the field \eqref{14}
in the initial coordinates $(x,u,p)$. Therefore, the $p$-axis and the $u$-axis are the only integral curves
of the field \eqref{14} passing through the origin.
The map $\pi \circ \Psi (x,u,p) = (x,y)$
sends the $p$-axis and $u$-axis in the $(x,u,p)$-space to the origin in the $(x,y)$-plane.
\end{proof}

We consider now the admissible variable $u=1$ and study the phase portrait of the field \eqref{14}
in a neighborhood of its singular point $(x,u,p) = (0,1,0)$.

Dividing the field \eqref{14} by the non-vanishing germ $(2uN_1 + \ldots)$
and making the affine change of variable $v=u-1$, we transform \eqref{14} into
\begin{equation}
p A(x,v,p) \pa_x + \bigl(\tfrac{1}{2}pB_1(x,v,p) - \eps x B_2(x,v,p)\bigr)\pa_p + v \pa_v,
\label{16}
\end{equation}
where $A$ and $B_{1}, B_{2}$ are smooth functions such that
$$
A(x,0,0) \equiv B_{1}(x,0,0) \equiv B_{2}(x,0,0) \equiv 1.
$$
It is worth observing that \eqref{16} has an isolated singularity at the origin
in the cases $D_s,D_n,D_f$ (when $\eps \neq 0$)
and a non-isolated singularity in the case $Z$ (when $\eps = 0$).
We deal with the cases $D_s,D_n,D_f$ and $Z$ separately, and start with the case $Z$.

%%%%%%%%%%%%%%%%%%%%%%%%%%%%%%%%%%%%%%%%%%%%%%%%%%%%%%%%%%%%%%%%%%%%%%%%

\subsection{The case $Z$}\label{Z}

\begin{theorem}
\label{T3}
Let $S$ be a surface with a pseudo-Riemannian metric \eqref{1} and let $q \in \D$ be a point satisfying
the conditions $Z$ in Table~\ref{Tab1}.
Then to the isotropic direction $p_0$ corresponds a one-parameter family $\Gamo$ of smooth geodesics
outgoing from $q$ into the Lorenzian domain, while
there are no geodesics outgoing from $q$ into the Riemannian domain.

There exist smooth local coordinates centered at the origin such that
the metric has the form \eqref{11}
and the geodesics $\gamap \in \Gamo$ outgoing from the origin are
\begin{equation}
y = \frac{1}{4}x^2 + \alf x^4 (1 + Y_{\alf}(x)), \ \ \,  Y_{\alf}(0)=0, \ \  \alf \in \mathbb R.
\label{17}
\end{equation}
The family \eqref{17} contains all three types of geodesics:
timelike if $\alf>0$, spacelike if $\alf<0$ and isotropic if $\alf=0$;
see {\rm Figure \ref{pic3}}, right.
\end{theorem}

\begin{proof}
The germ \eqref{16} with $\eps=0$ at the origin in the $(x,v,p)$-space has the form \eqref{48} with $n=3$ (Appendix~A).
Indeed, its components belong to the ideal $I = \<pA,v\> = \<p,v\>$
and the center manifold $W^c$, given by the equations $p=v=0$, consists of its singular points.
At any singular point sufficiently close to the origin, the spectrum of the linear part of \eqref{16} is $(1,\frac{1}{2},0)$.
Then, by Theorem \ref{PT5} (Appendix~A), the germ \eqref{16} at the origin is orbitally smoothly equivalent to
\begin{equation}
(2\xi + \phi(\zeta)\eta^2)\pa_{\xi} + \eta \pa_{\eta} .
\label{18}
\end{equation}

The restriction of the field \eqref{18} to each invariant leaf $\zeta = \const$ is a resonant node
with eigenvalues $2:1$ and associated eigenvectors $\pa_{v}, 2\pa_{x}+\pa_{p}$.
It has an infinite family $\Phi$ of integral curves all with the common tangent direction $2\pa_{x}+\pa_{p}$
and a unique integral curve with the tangent direction $\pa_{v}$, which coincides with the $v$-axis
and does not correspond to a geodesic.

By Lemma \ref{PL1} (Appendix~A), to prove that $\phi(\zeta)$ vanish on each leaf $\zeta = \const$, it is
sufficient to produce a smooth integral curve of the family $\Phi$. We use the integral curve $x=2p$, $v=0$
corresponding to the isotropic geodesic $y=\frac{1}{4}x^2$
which is a solution of  equation $p^2 = y$.
Hence $\phi(\zeta) \equiv 0$, and the normal form \eqref{18} becomes
\begin{equation}
2\xi \pa_{\xi} + \eta \pa_{\eta}.
\label{19}
\end{equation}

Comparing \eqref{16} with $\eps=0$ and \eqref{18}, one can show that
the conjugating diffeomorphism $(x,v,p) \mapsto (\xi, \eta, \zeta)$ can be chosen in the form
\begin{equation*}
\xi =v, \ \  \eta=pf(x,v,p), \ \  \zeta = x-2p +g(x,v,p), \ \ f(0)=1, \ \ \, g \in \MM^1,
%\label{20}
\end{equation*}
where $\MM^k$, $k \ge 0$, is the ideal of $k$-flat functions in the ring of smooth functions.

The field \eqref{19} has the invariant foliation $\xi = c\eta^2$, $c \in \bR P^1 = \bR \cup \infty$.
Returning to the coordinates $(x,v,p)$, we get the invariant foliation of the field \eqref{16} given by
\begin{equation}
v = cp^2 f^2(x,v,p), \ \ f(0)=1.
\label{21}
\end{equation}

The value $c = \infty$ corresponds to the invariant plane $\Pi = \{p=0\}$ which does not give any geodesics.
For $c \in \bR$, solving \eqref{21} in $v$ in a neighborhood of the origin gives $v = cp^2 \psi_c(x,p)$, where $\psi_c(x,p)$ is a smooth function with $\psi_c(0,0)=1$.
Returning to the initial variables $(x,y,p)$
by mean of the map $\Psi$ (formula~\eqref{13}),
we get the invariant foliation of the field \eqref{12}:
\begin{equation}
y = p^2 (1+cp^2 \psi_c(x,p)), \ \  \ \psi_c(0,0)=1,
\label{22}
\end{equation}
see Figure~\ref{pic3} (left).
The value $c=0$ corresponds to the isotropic surface $p^2 = y$.

Observe that all invariant leaves of the foliation~\eqref{22} belong to the semiplane $y>0$.
The plane $y=0$, which does not belong to the foliation~\eqref{22}, is an invariant plane of the field \eqref{12}
foliated by the vertical lines, which are integral curves of \eqref{12}.

After dividing by the common factor $p^2$, the restriction of the field \eqref{16}
to each invariant leaf \eqref{22} is given by
\begin{equation}
\frac{dp}{dx} = \frac{\ti \om - cp^2 \psi_c(\ti \om +p\ti \om_x +y\ti \om_y)}{2\ti \om (1+cp^2 \psi_c)} =
\frac{1}{2} -  cp^2 (1+ h_c(x,p)), \ \ \ h_c \in \MM^0,
\label{23}
\end{equation}
where $\ti \om$, $\ti \om_x$, $\ti \om_y$ are the restrictions of $\om$, $\om_x$, $\om_y$ to the
given leaf.
In a neighborhood of the origin, \eqref{23} defines a smooth and regular direction field,
and every leaf \eqref{22} contains a unique smooth integral curve of the field \eqref{16}
passing through the origin in the $(x,y,p)$-space with the tangent direction $\pa_x$,
see Figure~\ref{pic3} (center).
Projecting these integral curves to the $(x,y)$-plane (equivalently, integrating equation \eqref{23}),
we get a one-parameter family of geodesics in the statement of the theorem with $\alf = -\frac{c}{48}$,
see Figure \ref{pic3} (right).
\end{proof}

%%%%%%%%%%%%%%%%
\begin{figure}[ht]
\begin{center}
\includegraphics[height=4.5cm]{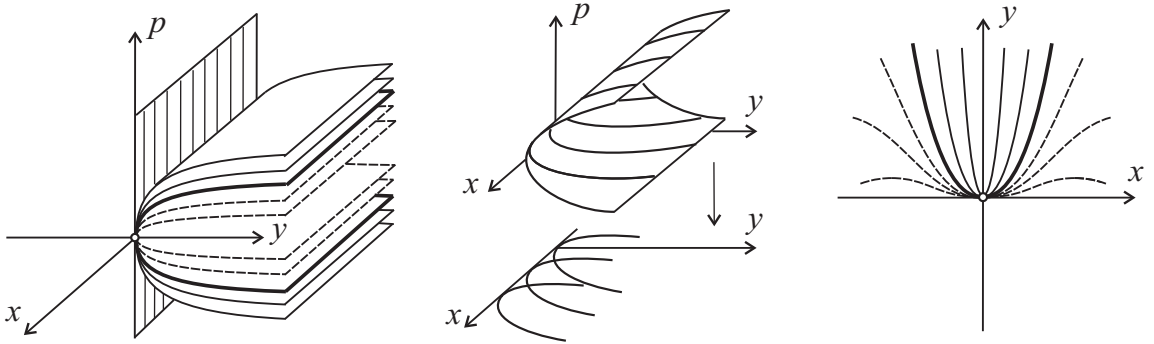}
\end{center}
\vspace{-1ex}
\caption{The case $Z$: invariant foliation \eqref{22} of the field \eqref{12} (left),
integral curves of \eqref{12} on an invariant leaf (center) and geodesics $\gamap \in \Gamo$ (right).
Here the discriminant curve $\D$ coincides with the $x$-axis.
Timelike, spacelike and isotropic geodesics are solid, dashed and bold solid lines, respectively.
}
\label{pic3}
\end{figure}

To illustrate Theorem~\ref{T3}, we consider the simplest example $ds^2 =  dy^2 - y dx^2$.
In this case, equation \eqref{5} can be studied using qualitative methods,
see, for example, \cite{Rem15} (Section~3).
The Lagrangian of the length functional $L=\sqrt{p^2-y}$ does not depend on the variable $x$,
hence the field \eqref{5} possesses the energy integral $H = L-pL_p$.
After evident transformations, equation $H = \const$ can be reduced to
\begin{equation}
p^2 = y - \alf y^2, \ \ \, \alf \in \bR,
\label{43000}
\end{equation}
which is a family of implicit differential equations of Clairaut type.

Every (unparametrized) geodesic in the metric $ds^2 =  dy^2 - y dx^2$ is a solution of equation \eqref{43000}.
Conversely, every solution of \eqref{43000} is a geodesic except the horizontal lines $y \equiv \const$,
each of which is the envelop of the family of integral curves of \eqref{43000} for a given $\alf$
(see \cite{Rem15}).
For instance, the value $\alf=0$ corresponds to the isotropic surface $p^2 = y$ (a parabolic cylinder)
and gives the isotropic geodesic $y=\frac{1}{4}x^2$ passing through the origin.
For determining non-isotropic geodesics, observe that every invariant surface \eqref{43000}
is a cylinder whose generatrices are parallel to the $x$-axis
and the base is an ellipse (if $\alf>0$) or a hyperbola (if $\alf<0$).
In the last case, the hyperbolic cylinder $p^2 = y-\alf y^2$ consists of two connected components:
{\it positive} and {\it negative} lying in the domains $y \ge 0$ and $y \le \alf^{-1}$, respectively.
See Figure \ref{pic5}, left.

Positive components of the hyperbolic cylinders ($\alf<0$) together with all other cylinders ($\alf \ge 0$)
form an invariant foliation of the form \eqref{21} (Figure \ref{pic5}, left).
Negative components of the hyperbolic cylinders do not intersect the plane $y=0$, and consequently, they do not contain
integral curves whose projections to the $(x,y)$-plane are geodesics passing through the $x$-axis.
Moreover, from \eqref{43000} it follows that $u=y/p^2 = 1/(1-\alf y)$, which shows that
$u \to 1$ along any geodesic
tending to the $x$-axis with isotropic tangential direction.
This gives a graphical explanation why $u=1$ is the only admissible value of the variable $u$ at $y=p=0$.

Thus to every $\alf \ge 0$ corresponds a geodesic $\gamap \in \Gamo$ which is timelike if $\alf>0$ or isotropic if $\alf=0$.
To every $\alf < 0$ corresponds a spacelike geodesic $\gamap \in \Gamo$, whose lift belongs
to the positive component of the hyperbolic cylinder $p^2 = y-\alf y^2$.
In contrast to this, the negative component of the same cylinder is filled with integral curves
of the field \eqref{5} whose projections on the $(x,y)$-plane are separated from the $x$-axis by
the horizontal strip $\alf^{-1}<y<0$. See Figure \ref{pic5}, right.

\begin{figure}[ht]
\begin{center}
\includegraphics[height=4.6cm]{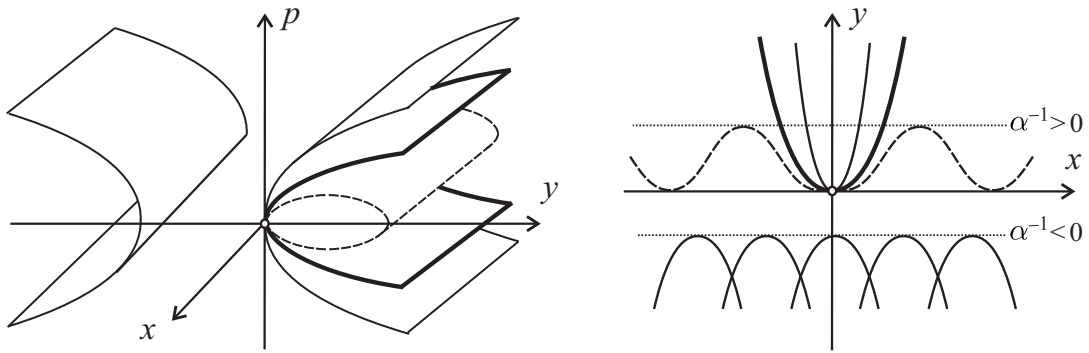}
\end{center}
\vspace{-1ex}
\caption{
The invariant foliation $p^2 = y-\alf y^2$ in the $(x,y,p)$-space (left)
and the corresponding geodesics (right).
Timelike, spacelike and isotropic geodesics are solid, dashed and bold solid lines, respectively.
}
\label{pic5}
\end{figure}

%%%%%%%%%%%%%%%%%%%%%%%%%%%%%%%%%%%%%%%%%%%%%%%%%%%%%%%%%%%%%%%%%%%%%%%%

\subsection{The cases $D_s, D_n, D_f$}\label{Dsnf}

We take the metric as in \eqref{11} with $\eps \neq 0, \frac{1}{16}$ near the point $q \in \D$.
The germ \eqref{16} with $\eps \neq 0$ at the origin is hyperbolic
with the spectrum $(1, \eps_1,\eps_2)$,
where $(\eps_1,\eps_2)$ is the spectrum of the field \eqref{4} at the origin, so
$\eps_1 + \eps_2 = \frac{1}{2}$ and $\eps_1 \eps_2 = \eps$.
The corresponding eigenvectors are $\pa_v$, $\pa_x+\eps_1\pa_p$ and $\pa_x+\eps_2\pa_p$.

Recall that we assumed that, in the cases $D_n, D_s$,
there are no non-trivial resonances \eqref{10} between $\eps_{1}, \eps_{2}, 1$.
Hence, by Theorem~\ref{PT1} (Appendix~A), the germ \eqref{16} at the origin is smoothly equivalent to
\begin{equation}
\xi \pa_{\xi} + \eps_1 \eta \pa_{\eta} + \eps_2  \zeta \pa_{\zeta}.
\label{24}
\end{equation}

\begin{remark}
{\rm
It follows by comparing \eqref{16} and  \eqref{24} that
the conjugating diffeomorphism $(x,v,p) \mapsto (\xi, \eta, \zeta)$ and its inverse
can be chosen in the form
\begin{equation}
\begin{aligned}
&
x = \eta (1+\phi_1) + \zeta (1+\phi_2), \ \ \ \,
p = \eta (\eps_1+\phi_3) + \zeta (\eps_2+\phi_4), \ \
v = \xi;
\\
&
\eta = \frac{x (\eps_2 +\psi_1) - p (1+\psi_2)}{\eps_2-\eps_1}, \ \
\zeta = \frac{x (\eps_1 +\psi_3) - p (1+\psi_4)}{\eps_1-\eps_2}, \ \
\xi  = v;
\end{aligned}
\label{25}
\end{equation}
where $\phi_i=\phi_i (\xi,\eta,\zeta)$ and $\psi_i=\psi_i(x,v,p)$ are smooth functions
vanishing at the origin.
}
\label{Rem3}
\end{remark}

\begin{theorem}
\label{T4}
Let $S$ be a surface with a pseudo-Riemannian metric \eqref{1} and let $q \in \D$ be a point satisfying
the conditions $D_s$ in Table~\ref{Tab1}.
Then to the isotropic direction $p_0$ corresponds a one-parameter family $\Gamo$ of
$C^2$-smooth geodesics outgoing from $q$ into the Lorenzian domain, while
there are no geodesics outgoing from $q$ into the Riemannian domain.

There exist smooth local coordinates centered at the origin such that
the metric has the form \eqref{11}
and the geodesics $\gamap \in \Gamo$ outgoing from the origin are
\begin{equation}
y = \frac{\eps_1}{2}x^2 + Y_{\alf}(x), \ \ \,  Y_{\alf}(x)=o(x^2), \ \ \eps_1 > \frac{1}{2}, \ \  \alf \in \mathbb R,
\label{26}
\end{equation}
together with the isotropic geodesic
\begin{equation}
y = \frac{\eps_2}{2}x^2 + Y(x), \ \ \,  Y(x)=o(x^2), \ \ \eps_2 < 0.
\label{27}
\end{equation}
The family \eqref{26} contains all three types of geodesics:
timelike if $\alf<0$, spacelike if $\alf>0$ and isotropic if $\alf=0$;
see {\rm Figure \ref{pic6}}, right.
\end{theorem}

\begin{proof}
By Lemma~\ref{PL2} (Appendix~A), equation
$c_1 |\eta|^{1/\eps_1} + c_2 |\zeta|^{1/\eps_2} + c_3 \xi = 0$
with any constants $c_i$ defines an invariant surface of the field \eqref{24}.
Since $\eps<0$, the eigenvalues $\eps_{1}$ and $\eps_{2}$ have different signs.
We assume, without loss of generality, that  $\eps_1>\frac{1}{2}$ and $\eps_2<0$.
Hence, the invariant surface $c_1 |\eta|^{1/\eps_1} + c_2 |\zeta|^{1/\eps_2} + c_3 \xi = 0$
passes through the origin if and only if $c_1=c_3=0$ or $c_2=0$. This yields the invariant foliation
\begin{equation}
\xi = \alf |\eta|^{1/\eps_1}, \ \ \alf \in \bR P^1 = \bR \cup \infty,
\label{28}
\end{equation}
where $\alf = \infty$ corresponds to the plane $\eta=0$ being the limit set of
the cylindric surfaces $\xi = \alf |\eta|^{1/\eps_1}$ as $\alf \to \infty$.
The restriction of the field \eqref{24} to each invariant leaf \eqref{28} is a saddle with the separatrices
\begin{equation*}
\begin{aligned}
\{\xi=\eta=0\} \ \, &\textrm{and} \ \,  \{\xi = \alf |\eta|^{1/\eps_1}, \ \zeta=0\} \ \
\textrm{for} \ \,  \alf \neq \infty,
\\
\{\xi=\eta=0\} \ \, &\textrm{and} \ \, \{\eta=\zeta =0\} \ \ \textrm{for} \ \, \alf = \infty,
\end{aligned}
\end{equation*}
where the separatrix $\{\xi=\eta=0\}$ is the intersection of all invariant leaves \eqref{28}.
The corresponding invariant foliation of the field \eqref{12} is as in Figure \ref{pic6} (left).

Taking into account \eqref{25}, $u=v+1$ and $y = \eps x^2 +up^2$, one can see that
the separatrix $\{\eta=\zeta =0\}$ corresponds
to the $v$-axis in the $(x,v,p)$-space and to the origin in the $(x,y,p)$-space.
The remaining separatrices $\{\xi = \alf |\eta|^{1/\eps_1}, \zeta=0\}$ and $\{\xi=\eta=0\}$
yield the family of integral curves of the field \eqref{14}
passing through the point $(0,1,0)$ in the $(x,u,p)$-space:
\begin{align}
\label{29}
\{ x=\eta(1 + & \ti \phi_{1,\alf}(\eta)), \ p=\eta(\eps_1+\ti \phi_{3,\alf}(\eta)), \ u= 1+\alf |\eta|^{1/\eps_1} \},
\ \ \, \alf \in \bR,
\\
\label{30}
&\{ x=\zeta(1+\ti \phi_{2}(\zeta)), \ p=\zeta(\eps_2+\ti \phi_{4}(\zeta)), \ u= 1\},
\end{align}
where
$\ti \phi_{i,\alf}(\eta)=\phi_{i}(\alf |\eta|^{1/\eps_1},\eta,0)$ for $i=1,3$, and
$\ti \phi_{i}(\zeta)=\phi_{i}(0,0,\zeta)$ for $i=2,4$.
Observe that the integral curve \eqref{29} with $\alf=0$ and the integral curve \eqref{30}
belong to the isotropic surface $\{u=1\}$.

Since the right-hand sides of $x=\eta(1+\ti \phi_{1,\alf}(\eta))$ and $x=\zeta(1+\ti \phi_{2}(\zeta))$
are $C^1$ and $\Ci$-smooth respectively, both equations can be solved for $\eta$.
Hence, from \eqref{29} and \eqref{30} one can express the variables $p$ and $u$ as functions of $x$.
To complete the proof, observe that the map $\pi \circ \Psi (x,u,p) = (x,y)$
sends the family of integral curves \eqref{29} to the family of geodesics \eqref{26}
and the integral curve \eqref{30} to the geodesic \eqref{27}; see Figure \ref{pic6} (right).
\end{proof}

\begin{remark}
{\rm
In the case $Z$ the family $\Gamo$ given by \eqref{17} contains a unique isotropic geodesic $y=\frac{1}{4}x^2$
(the bold line in Figure~\ref{pic3}, right) which separates $\Gamo$ into timelike and spacelike geodesics.
In the case $D_s$ the family $\Gamo$  contains two isotropic geodesics, corresponding to
the integral curve \eqref{29} with $\alf=0$ and the integral curve \eqref{30}.
The first (the bold line in Figure~\ref{pic6}, right)
separates $\Gamo$ into timelike and spacelike geodesics,
while the second (the double line in Figure~\ref{pic6}, right) does not.
The difference between the families $\Gamo$ in the cases $Z$ and $D_s$ is related to
topological structure of the invariant foliations of the field \eqref{12}
in the cases $Z$ and $D_s$; compare Figures \ref{pic3} (left) and \ref{pic6} (left).
In the case $Z$ all invariant leaves intersect on the line $\Psi (\Pi)$ only, while
in the case $D_s$ they intersect on the dotted line $\Psi (\Pi)$
and on the double line which corresponds to $\{\xi=\eta=0\}$ in the normal form \eqref{24}.
}
\label{Rem4}
\end{remark}

%%%%%%%%%%%%%%%%

\begin{figure}[ht]
\begin{center}
\includegraphics[height=4.8cm]{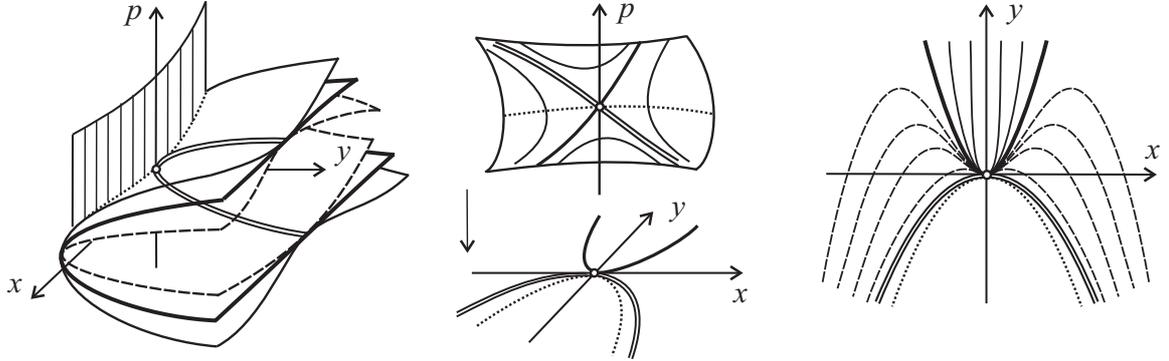}
\end{center}
\vspace{-1ex}
\caption{The case $D_s$: invariant foliation of the field \eqref{12} (left),
integral curves of \eqref{12} on an invariant leaf (center) and geodesics $\gamap \in \Gamo$ (right).
On the right: timelike and spacelike geodesics are solid and dashed lines, respectively.
The bold line and the double line are isotropic geodesics, the dotted line is the discriminant curve.
}
\label{pic6}
\end{figure}

\newpage

\begin{theorem}
\label{T5}
Let $S$ be a surface with a pseudo-Riemannian metric \eqref{1} and let $q \in \D$ be a point satisfying
the conditions $D_n$ or $D_f$ in Table~\ref{Tab1}.
Then to the isotropic direction $p_0$ corresponds a two-parameter family $\Gamo$ of
$C^2$-smooth geodesics outgoing from $q$ into the Lorenzian domain, while
there are no geodesics outgoing from $q$ into the Riemannian domain.

There exist smooth local coordinates centered at the origin such that
the metric has the form \eqref{11}
and the field \eqref{12} has the invariant foliation
\begin{equation}
y = \eps x^2 + p^2(1 + \alf Y_{\alf}(x,p)), \ \ Y_{\alf} \in \MM^2, \ \  \alf \in \bR,
\label{31}
\end{equation}
see Figure \ref{pic7}, left.
The restriction of the field \eqref{12} to every invariant leaf \eqref{31} is a node {\rm (}for $D_n${\rm )}
or a focus {\rm (}for $D_f${\rm )} with spectrum $\eps_1,\eps_2$, and the geodesics $\gamabp \in \Gamo$
are projections of the integral curves of \eqref{12} on the leaves \eqref{31}
to the $(x,y)$-plane {\rm (Figure \ref{pic7}, right)}.
The geodesics $\gamabp \in \Gamo$ are
timelike if $\alf<0$, spacelike if $\alf>0$ and isotropic if $\alf=0$.
\end{theorem}

\begin{figure}[ht]
\begin{center}
\includegraphics[height=5.2cm]{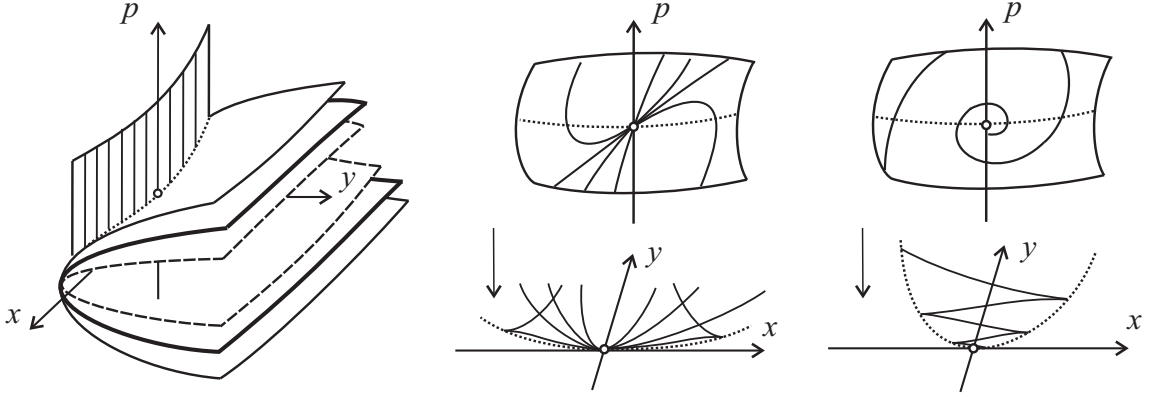}
\end{center}
\vspace{-1ex}
\caption{The cases $D_n$ and $D_f$:
invariant foliation of the field \eqref{12} (left), 
integral curves of the field \eqref{12} on an invariant leaf and their projection
to the $(x,y)$-plane in the cases $D_n$ (center) and $D_f$ (right).
}
\label{pic7}
\end{figure}

\begin{proof}
We start with the question of linearizability of the germ \eqref{16} at the origin.
In the case $D_n$, the germ \eqref{16} is smoothly equivalent to the linear normal form \eqref{24}
with $0 < \eps_{1}, \eps_{2} < \frac{1}{2}$ (Theorem~\ref{PT1} in Appendix~A),
and the conjugating diffeomorphism has the form \eqref{25}.

In the case $D_f$, Theorem~\ref{PT1} does not apply because the spectrum $(1, \eps_1,\eps_2)$ with
\begin{equation}
\eps_{1}=a +bi,\ \ \eps_2 = a - i b,  \ \  a = \tfrac{1}{4}, \ \  b= \tfrac{1}{4}\sqrt{16\eps - 1}, \ \
\eps = \eps_{1}\eps_{2} > \tfrac{1}{16},
\label{32}
\end{equation}
has the resonance $2(\eps_1+\eps_2) = 1$ of the order $|s|=4$ impeding the linearizability
in the $C^4$-smooth category.
Consider along with \eqref{10} the resonances obtained by taking the real part of both sides of \eqref{10}.
The spectrum \eqref{32} has five resonances of this type of order greater than one:
$\Real (s_1 \eps_1 + s_2 \eps_2) = 1$, $|s|=s_1+s_2=4$.
Thus in the case $D_f$, the germ \eqref{16} satisfies the condition of Theorem~\ref{PT3} (Appendix~A) with $k=1$,
and consequently, it is $C^1$-smoothly equivalent to
\begin{equation}
\xi \pa_{\xi} +
(a \eta + b \zeta) \pa_{\eta} + (a \zeta - b\eta) \pa_{\zeta}
\label{33}
\end{equation}
with $a,b$ as in \eqref{32}.

Thus in the case $D_n$ (resp. $D_f$) the germ \eqref{16} is $\Ci$ (resp. $C^1$) smoothly equivalent to
the linear normal form \eqref{24} (resp. \eqref{33}), and the conjugating diffeomorphism has the form
\begin{equation}
\begin{aligned}
&
x = \eta \phi_1 + \zeta \phi_2, \ \
p = \eta \phi_3 + \zeta \phi_4, \ \
v = \xi,
\\
&
\eta = x \psi_1 + p \psi_2, \ \
\zeta = x \psi_3 + p \psi_4, \ \
\xi  = v,
\end{aligned}
\label{34}
\end{equation}
for some $\Ci$ (resp. $C^1$) smooth functions $\phi_i=\phi_i (\xi,\eta,\zeta)$ and $\psi_i=\psi_i(x,v,p)$.

By Lemma \ref{PL2} (Appendix~A), equations
\begin{align}
\label{35}
\xi &= \alf (|\eta|^{1/\eps_1} + |\zeta|^{1/\eps_2}), \\
\label{36}
\xi &= \alf (\eta^2+\zeta^2)^2,
\end{align}
where $\alf \in \bR P^1$, define invariant foliations of the linear fields \eqref{24} and \eqref{33}, respectively.
Here $\alf=\infty$ gives the $\xi$-axis in the normal coordinates, i.e., the $v$-axis in the $(x,v,p)$-space,
and consequently, the origin in the $(x,y,p)$-space.
Hence we shall consider both equations \eqref{35} and \eqref{36} with $\alf \in \bR$ only.

Substituting expressions \eqref{34} for $\xi, \eta, \zeta$ in \eqref{35} and \eqref{36},
gives the invariant foliation of the field \eqref{16} in the case $D_n$ and $D_f$ respectively:
\begin{align}
\label{37}
v &= \alf \bigl(|x \psi_1 + p \psi_2|^{1/\eps_1} + |x \psi_3 + p \psi_4|^{1/\eps_2}\bigr), \\
\label{38}
v &= \alf \bigl(|x \psi_1 + p \psi_2|^2+ |x \psi_3 + p \psi_4|^2\bigr)^2.
\end{align}

Note that the right-hand sides of \eqref{37} and \eqref{38} are at least $C^1$-smooth
(in the case $D_n$, $0 < \eps_1, \eps_2 < \frac{1}{2}$).
By the implicit function theorem, \eqref{37} and \eqref{38} can be solved in $v$ near the origin.
In both cases we get the equation $v = \alf Y_{\alf}(x,p)$ with a $C^1$-smooth function $Y_{\alf}$ vanishing at the origin.
Substituting this expression into $y = \eps x^2 + (1+v)p^2$, gives the invariant foliation \eqref{31}
of the field \eqref{12}.

Similarly to the case $Z$, one can divide the restriction of the field \eqref{12} to every invariant leaf
by the common factor $p$ (see formula~\eqref{23}). Then the restriction of the field \eqref{12}
to every invariant leaf \eqref{31} becomes a node (in the case $D_n$) and a focus (in the case $D_f$)
with the spectrum $(\eps_1,\eps_2)$ (Figure \ref{pic7}, center, right).
\end{proof}

Figure \ref{pic8} presents computer generated pictures of geodesics on a surface endowed with the
metric \eqref{11} with $\om \equiv -1$ and $\Theta \equiv 0$ in the cases $D_s, D_n, D_f$.

\begin{figure}[ht]
\begin{center}
\includegraphics[height=5.5cm]{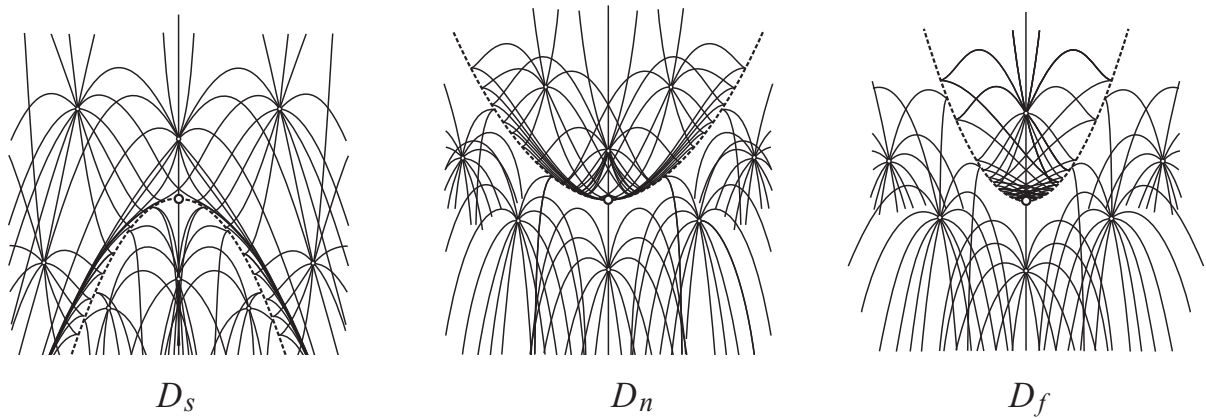}
\end{center}
\vspace{-1ex}
\caption{
Geodesics (solid lines) in the metric \eqref{11} with $\om \equiv -1$ and $\Theta \equiv 0$
in the cases $D_s$ (left), $D_n$ (center), $D_f$ (right).
The dotted line is the discriminant curve.
%Geodesics (black lines) in the metric \eqref{11} with $\om \equiv -1$ and $\Theta \equiv 0$
%in the cases $D_s$ (left), $D_n$ (center), $D_f$ (right).
%The red line is the discriminant curve.
}
\label{pic8}
\end{figure}

% % % % % % % % % % % % % % % % % % % % % % % % % % % % % % % % % % % % % % % % % % % % % % % % % % % % % % % % % % % % % % %
% % % % % % % % % % % % % % % % % % % % % % % % % % % % % % % % % % % % % % % % % % % % % % % % % % % % % % % % % % % % % % %
%                             A P P E N D I X :

\section{Appendix A. Local normal forms of vector fields}

Here we give a brief survey of local normal forms for vector fields in real phase space,
which were used in this paper (for more details, see also surveys in \cite{AI, JBG}).
All vector fields are supposed to be smooth
(that means $\Ci$ unless stated otherwise). By $\Com$ we denote
the class of analytic mappings.
For convenience, we also present vector fields as autonomous differential equations,
where the differentiation is by an auxiliary parameter playing a role of time.

\begin{defin}
{\rm
Two vector fields $V_1$ and $V_2$ are $C^k$-smoothly (resp. topologically) equivalent, if there exists a $C^k$-diffeomorphism (resp.  homeomorphism) $h \colon \bR^n \to \bR^n$ that conjugates their phase flows
$g^t_1$ and $g^t_2$, i.e., $h \circ g^t_1 = g^t_2 \circ h$.
Here $k$ is an integer number (finite-smooth equivalence) or $\infty$ (infinite-smooth equivalence)
or $\om$ (analytic equivalence).
}
\end{defin}

\begin{defin}
{\rm
Two vector fields $V_1$ and $V_2$ are {\it orbitally} $C^k$-smoothly (resp. topologically) equivalent,
if there exists a $C^k$-diffeomorphism (resp.  homeomorphism) that conjugates
their integral curves (orbits of their phase flows).
}
\end{defin}

\begin{remark}
{\rm
The second definition slightly differs from the generally accepted definition of the orbital equivalence,
where coincidence of the orientation of integral curves is also required.
In fact, our definition is naturally related to directions fields, whose integral curves
do not have a orientation a priori.
}
\label{RemA}
\end{remark}

A great deal of work is done to bring a vector field to a normal form
in a neighborhood of its singular point under a chosen equivalence relation.
In particular, the germ of a vector field at its singular point
is called  {\it linearizable} (in a certain category)
if it is equivalent (in this category) to its linear part.
Of course, not all vector fields are linearizable,
and when it is not, the question is what normal formal can we get.

\begin{defin}
{\rm
Let $V$ be the germ of a vector field with singular point the origin $0$ in $\bR^n$ and
$\la = (\la_1, \ldots, \la_n) \in \bC^n$ be the spectrum of $V$ at $0$.
The germ $V$ is called {\it hyperbolic} if the spectrum $\la$ does not contain neither zeros nor pure imaginary eigenvalues. The germ $V$ is called {\it partially hyperbolic} if
the spectrum $\la$ contains at least one eigenvalue with non-zero real part.
}
\end{defin}

\subsection{Hyperbolic germs}

Let $V$ be the germ of a  hyperbolic vector field at $0$.
For topological equivalence the only local invariants are the scalars $a_i = \sign (\Real \la_i)$.
The Grobman-Hartman Theorem \cite{AI, Hartman} states that $V$ is topologically equivalent to the field
%$a_1 \xi_1 \pa_{\xi_1} + \cdots + a_n \xi_n \pa_{\xi_n}$.
$\dot \xi_i = a_i \xi_i$, $i=1, \ldots, n$.

For the $C^k$-smooth equivalence, resonances of two types play an important role:
\begin{align}
\label{39}
\la_j - (s,\la) = 0,  \ \ s_i \in \bZ_+, \ \  j\in \{1,\ldots, n\}, \\
\label{40}
\Real (\la_j - (s,\la)) = 0, \ \ s_i \in \bZ_+, \ \  j\in \{1,\ldots, n\},
\end{align}
where $(s,\la) = s_1\la_1 + \cdots + s_n\la_n$ is the standard scalar product.
In both relations \eqref{39}, \eqref{40} the natural number
$|s| = s_1 + \cdots +  s_n$ is called the {\it order} of the resonance, and
$\xi^s = \xi_1^{s_1} \cdots \xi_n^{s_n}$ is called the {\it resonant monomial}.
In general, non-trivial resonances \eqref{39}, \eqref{40} of orders $|s|\ge 2$
are obstacles for the germ $V$ to be $C^k$-smoothly linearizable.

\begin{theorem}[\cite{AI, Chen, Sternberg}]
If the spectrum $\la$ does not have resonances \eqref{39} of any order $|s| \ge 2$,
the germ $V$ is $\Ci$-smoothly linearizable.
Moreover, if $\la \in \bR^n$ and has only trivial resonances \eqref{39},
the germ $V$ is $\Ci$-smoothly equivalent to
%the diagonal linear field
\begin{equation}
%\la_1 \xi_1 \pa_{\xi_1} + \cdots + \la_n \xi_n \pa_{\xi_n}.
\dot \xi_i = \la_i \xi_i, \ \ \ i=1, \ldots, n.
\label{41}
\end{equation}
\label{PT1}
\end{theorem}

\vspace{-4ex}

\begin{theorem}[\cite{AI, Chen, Hartman}]
If $\la \in \bR^n$, then the germ $V$ is $\Ci$-smoothly equivalent to the field
\begin{equation}
%(l_1(\xi) + \om_1(\xi)) \pa_{\xi_1} + \cdots + (l_n(\xi) + \om_n(\xi)) \pa_{\xi_n},
\dot \xi_i = \la_i \xi_i + \sum_{s \in \bZ_+^n} {a_{is}} \xi^s, \ \ \ i=1, \ldots, n.
\label{42}
\end{equation}
where ${a_{is}} \neq 0$ only if the resonance \eqref{39} holds.
\label{PT2}
\end{theorem}

In other words, Theorems \ref{PT1} and \ref{PT2} states that the normal forms \eqref{41}, \eqref{42}
contain resonant monomials only.
Here the linear terms $\la_i \xi_i$ correspond to trivial resonances \eqref{39}, while all remaining terms
${a_{is}} \xi^s$ correspond to non-trivial resonances \eqref{39} if they exist.
Under some additional restrictions, Theorems \ref{PT1}, \ref{PT2} are valid in $\Com$ category.
For instance, in the case $\la \in \bR^n$ it is sufficient to require that all $\la_i$ have the same sign.
However, if $\la_i$ have different signs, the restrictions are much stronger, especially
if $\la$ has non-trivial resonances (Theorem~\ref{PT2}), see \cite{AI, JBG} and the references therein.

On the contrary, requirements become weaker if we consider $C^k$-smooth equivalence with $k<\infty$,
see e.g. \cite{AI, Sam82, Sam96}. In this paper, we need the following result.

\begin{theorem}[\cite{AI, Sell}]
Suppose that for some positive integer $k$
the spectrum $\la$ has neither resonances \eqref{39} of order $2 \le |s|\le 2k$
nor resonances \eqref{40} of order $|s| = 2k$.
Then $V$ is $C^k$-smoothly linearizable.
\label{PT3}
\end{theorem}

We establish below the existence of invariant foliations of a hyperbolic linear vector field
in 3-dimensional real phase. Assume, without loss of generality,
that one of the eigenvalues $\la_i$ of is equal to 1.
It is sufficient to consider the following two fields
\begin{align}
\label{43}
%V_1&= \la_1 \xi_1 \pa_{\xi_1} + \la_2 \xi_2 \pa_{\xi_2} + \xi_3 \pa_{\xi_3},\\
\dot \xi_1 &= \la_1 \xi_1, \ \  \dot \xi_2 = \la_2 \xi_2, \ \  \dot \xi_3 = \xi_3,  \\
\label{44}
\dot \xi_1 &= (\alf \xi_1 + \beta \xi_2), \ \
\dot \xi_2 = (\alf \xi_2 - \beta \xi_1), \ \
\dot \xi_3 = \xi_3,
\end{align}
where $\la_{1},\la_{2}$ and $\alf, \beta$ are real and non-zero.

\begin{lemma}
For any real constants $c_i$ the equations
\begin{align*}
%\label{45}
c_1 |\xi_1|^{1/\la_1} + c_2 |\xi_2|^{1/\la_2} + c_3 \xi_3 &= 0,\\
%\label{46}
c_1 (\xi_1^2+\xi_2^2)^{1/2\alf} + c_3 \xi_3& = 0,
\end{align*}
define invariant surfaces of the fields \eqref{43} and \eqref{44}, respectively,
\label{PL2}
\end{lemma}
The proof is trivial and is omitted.

% % % % % % % % % % % % % % % % % % % % % % % % % % % % % % % % % % % % % % %

\subsection{Partially hyperbolic germs}

Let $W^s$, $W^u$, $W^c$  be the unstable, stable and center
manifold of the partial hyperbolic germ $V$ at $0$, and let $d_i = \dim W^i$, $i=s,u,c$.
Set $d=d_s+d_u$, then $d, d_c >0$ and $d+d_c=n$.
One can choose local coordinates
$(\xi_1, \ldots, \xi_{d_s}) \in W^s$, $(\xi_{d_s+1}, \ldots, \xi_{d}) \in W^u$,
$\zeta = (\zeta_1, \ldots, \zeta_{d_c}) \in W^c$.

\begin{theorem}[\cite{AI, HPS}]
The germ $V$ is topologically equivalent to the direct product of $d$-dimensional standard saddle
(the first $d$ equations)
and the restriction of $V$ to the center manifold (the last $d_c$ equations):
\begin{equation*}
\begin{aligned}
%\pa_{\xi_i} - \sum_{i=1}^{d_s} \eta_i \pa_{\eta_i} + \sum_{i=1}^{d_c} v_i(\zeta) \pa_{\zeta_i}.
\dot \xi_i = \xi_i, \ \, i = 1, &\ldots,d_s; \ \ \,
\dot \xi_i = -\xi_i, \ \, i = d_s+1,\ldots,d;  \\
&\dot \zeta_j = Z_j(\zeta), \ \, j = 1,\ldots,d_c.
\end{aligned}
%\label{47}
\end{equation*}
\label{PT4}
\end{theorem}

\medskip

In this paper, we deal with a special class of partially hyperbolic vector fields,
which were studied by many authors, see e.g. \cite{Rouss,Takens}.
From now on, we assume that all components of the germ $V$ belong to the ideal $I$
(in the ring of smooth functions vanishing at $0$) generated by two of them.

More specifically, such a germ $V$ has the form
\begin{equation}
%v \pa_{\xi} + w \pa_{\eta} + \sum_{i=1}^{n-2} (\alf_i v + \beta_i w) \pa_{\zeta_i},
\dot \xi = v, \ \
\dot \eta = w, \ \
\dot \zeta_j = \alf_j v + \beta_j w, \ \  j =1, \ldots, n-2,
\label{48}
\end{equation}
where $v,w$ and $\alf_j, \beta_j$ are smooth functions of the variables $\xi, \eta, \zeta_j$.
The components of the germ \eqref{48} belong to the ideal $I = \<v,w\>$,
and the spectrum of $V$ contains at most two non-zero eigenvalues: $\la = (\la_1, \la_2, 0, \ldots, 0)$.

We shall further assume that $\Real \la_{1}\neq 0$ and $\Real \la_{2} \neq 0$.
Hence the center manifold $W^c = \{ v=w=0 \} \subset \bR^n$ is a smooth manifold of codimension 2 and
the restriction of the field $V$ to $W^c$ is identically zero, so $W^c$ consists of singular points of $V$.
%Such vector fields are studied in \cite{GR, Rouss}.
By Theorem \ref{PT4}, the germ $V$ is topologically equivalent to
$$
\dot \xi = a_1 \xi, \ \ \ \dot \eta = a_2 \eta, \ \ \, \dot \zeta_j = 0, \ \  j =1, \ldots, n-2,
$$
where $a_i= \sign (\Real \la_i)$, $i=1,2$.

For $C^k$-smooth classification of the germ \eqref{48}, we need to introduce two types of resonances
between the non-zero eigenvalues $\la_{1}, \la_{2}$ being, in fact, partial cases of \eqref{39}:
\begin{eqnarray}
&&
s_1\la_1 + s_2\la_2 = 0, \ \ s_i \in \bZ_+, \ \  i=1,2,
\label{49}\\
&&
s_1\la_1 + s_2\la_2 = \la_j, \ \  s_i  \in \bZ_+, \ \  i,j=1,2.
\label{50}
\end{eqnarray}

For simplifying the presentation, further we shall always assume that $|\la_1| \ge |\la_2|$
and exclude from consideration trivial resonances \eqref{49} ($s_1=s_2=0$)
and \eqref{50} ($s_1=1$, $j=1$ or $s_2=1$, $j=2$).
Then the absence of resonances \eqref{50} implies the absence of \eqref{49}.
On the other hand, in the absence of \eqref{49}, resonances \eqref{50} may have only the form
$\la_1=m\la_2$, for positive integers $m$.

Given the germ \eqref{48} with $\Real \la_{1}\neq 0$ and $\Real \la_{2} \neq 0$
we choose local coordinates $\xi, \eta$ (called {\it hyperbolic variables})
and $\zeta = (\zeta_1, \ldots, \zeta_{n-2}) \in W^c$  (called {\it non-hyperbolic variables})
such that the ideal $I = \<\xi,\eta\>$, and consequently,
the center manifold $W^c$ is given by $\xi=\eta=0$.
The linearization of $V$ with respect to the hyperbolic variables
has two eigenvalues which are continuous functions $\la_{1}(\zeta)$ and  $\la_{2}(\zeta)$ of $\zeta \in W^c$.
We have $\la_{j}(0)=\la_j$, $j=1,2$ and $\la_j$ as above.

An analogue of Theorems \ref{PT1} and \ref{PT2} is the following.

\begin{theorem}[\cite{GR, Rouss}]
Suppose that between $\la_{1}(\zeta)$ and  $\la_{2}(\zeta)$ there are no non-trivial resonances \eqref{49}
of any order $|s| \ge 1$ for all $\zeta$ sufficiently close to zero.
Then the germ \eqref{48} is $\Ci$-smoothly equivalent to
\begin{equation}
%X_1(\xi,\eta,\zeta) \pa_{\xi} + X_2(\xi,\eta,\zeta) \pa_{\eta},
\dot \xi = X, \ \ \, \dot \eta = Y, \ \ \,
\dot \zeta_j = 0, \ \  j =1, \ldots, n-2,
\label{51}
\end{equation}
where $X,Y$ are smooth functions of $\xi, \eta, \zeta_j$ such that the ideal
$I = \<X,Y\> = \<\xi,\eta\>$.
Moreover, if in addition the eigenvalues $\la_{1}, \la_{2} \in \bR$, then
\begin{equation}
X = \la_1(\zeta)\xi + \phi(\zeta)\eta^m, \ \ Y = \la_2(\zeta)\eta,
\label{52}
\end{equation}
where $\phi(\zeta) \not\equiv 0$ only if $\la_{1} = m \la_{2}$ with some natural $m \ge 1$.
\label{PT5}
\end{theorem}

\begin{remark}
{\rm
If the pair $(\la_{1}, \la_{2})$ belongs to the Poincar\'e domain
(i.e., $\la_{1}$ and $\la_{2}$ are real and of the same sign or complex conjugate),
then the condition
\begin{equation}
s_1 \la_1(\zeta) + s_2\la_2(\zeta) \neq 0, \ \ \forall \, s_i \in \bZ_+, \ \  i=1,2, \ \ \forall \, \zeta \in W^c,
\label{53}
\end{equation}
follows from \eqref{49}. Moreover, in this case Theorem \ref{PT5} is valid in $\Com$ category.
However, if the pair $(\la_{1}, \la_{2})$ belongs to the Siegel domain
(i.e., $\la_{1}$ and $\la_{2}$ are real and of different signs),
the condition \eqref{53} is equivalent to
$\la_{1}(\zeta):\la_{2}(\zeta) \equiv \const$ for all $\zeta \in W^c$.
}
\label{RemB}
\end{remark}

The conditions in Theorem \ref{PT5} become weaker if we consider $C^k$-smooth equivalence with $k<\infty$.
Set
\begin{equation*}
N(k) = 2 \biggl[ (2k+1) \frac{\max |\Real \lambda_{1,2}|}{\min \, |\Real \lambda_{1,2}|} \, \biggr] + 2,
\quad k \in \mathbb N,
\end{equation*}
where the square brackets is the integer part of a number.

\begin{theorem}[\cite{GR, Sam82}]
For any $k\in \bN$, the statements in Theorem \ref{PT5} still hold true
if $\Ci$ is replaced with $C^k$ and the inequalities
$1 \le |s|$, $1 \le m$ are replaced with $1 \le |s| \le N(k)$, $1 \le m \le N(k)$ respectively.
\label{PT55}
\end{theorem}

The normal form \eqref{51}, \eqref{52} can be further simplified.
For our purposes, we are interested in orbital normal form in the case when
the resonance $\la_1(\zeta) = m\la_2(\zeta)$ holds at all $\zeta \in W^c$.
Then, dividing by $\la_2(\zeta)$, from \eqref{51}, \eqref{52} we get the orbital normal form
\begin{equation}
\dot \xi = (m\xi + \psi(\zeta)\eta^m), \ \ \dot \eta = \eta, \ \ \,
\dot \zeta_j = 0, \ \  j =1, \ldots, n-2.
\label{54}
\end{equation}
where the smooth functions $\phi(\zeta)$ and $\psi(\zeta)$ vanish simultaneously.

The following lemma gives a simple geometric criterion for $\psi(\zeta) \equiv 0$,
which is important for applications.

\begin{lemma}
Let $V$ be the germ of a field from Theorem~\ref{PT5} with the normal form \eqref{51}, \eqref{52}
and the resonance $\la_1(\zeta) = m\la_2(\zeta)$, $m>1$, holds at all points $\zeta \in W^c$.
Then in the orbital normal form \eqref{54}, $\psi(\zeta) = 0$ if and only if
$V$ has a $C^{m}$-smooth integral curve that passes through
the corresponding point $\zeta \in W^c$ with the tangential direction
parallel to the eigenvector with $\la_2(\zeta)$.
\label{PL1}
\end{lemma}

\begin{proof}
The field  \eqref{54} can be integrated explicitly.
It has the invariant foliation $\zeta = \const$ and each leaf
contains a single integral curve $\eta=0$ with tangential direction $\pa_{\xi}$
and one-parameter family of integral curves
\begin{equation}
\xi = \eta^m (c+\psi(\zeta)\ln |\eta|), \ \ c=\const,
\label{55}
\end{equation}
with the common tangential direction $\pa_{\eta}$ at the point $\xi=\eta=0$.
All the curves \eqref{55} are $C^{m-1}$-smooth at $\xi=\eta=0$
if $\phi(\zeta) \neq 0$ and $\Ci$-smooth at zero if $\phi(\zeta)\equiv 0$.

Given $\zeta \in W^c$, the existence of at least one $C^{m}$-smooth integral curve passing through the point $\xi=\eta=0$ (the intersection of $W^c$ with the corresponding invariant leaf $\zeta = \const$)
with the tangential direction parallel to the eigenvector with $\la_2(\zeta)$
is equivalent to the condition $\psi(0)=0$.
\end{proof}

\begin{theorem}[\cite{GR, Rouss}]
Suppose that the resonance $\la_1(\zeta) + \la_2(\zeta)=0$ holds at all singular points
$\zeta \in W^c$ and $\Real \la_{1} \neq 0$, $\Real \la_{2} \neq 0$.
Then, for any natural $k$, the germ $V$ is $C^k$-smoothly equivalent to
\begin{equation}
\begin{aligned}
\dot \xi = \xi(\lambda_1(&\zeta) + \rho \Phi_1(\rho,\zeta)),  \quad
\dot \eta = \eta(\lambda_2(\zeta) + \rho \Phi_2(\rho,\zeta))_, \\
&\dot \zeta_j = \rho \Psi_j(\rho,\zeta), \, \quad j=1,\ldots, n-2,
\end{aligned}
\label{4.15}
\end{equation}
where $\Phi_i(\rho,\zeta)$ and $\Psi_j(\rho,\zeta)$ are polynomials in $\rho = \xi \eta$ of degrees $N(k)-1$.

If $\Psi_j(0,0) \neq 0$ for at least one $j=1,\ldots, n-2$, then
the germ $V$ is $\Ci$-smoothly orbitally equivalent to
\begin{equation}
\dot \xi = \xi, \ \ \dot \eta = -\eta, \ \ \dot \zeta = \xi\eta, \, \quad j=1,\ldots, n-2.
\label{4.16}
\end{equation}
\label{PT6}
\end{theorem}

Theorem \ref{PT6} is not valid in $\Com$ category.

% % % % % % % % % % % % % % % % % % % % % % % % % % % % % % % % % % % % % % % % % % % % % % % % % % % % % % % % % % % % % % %
% % % % % % % % % % % % % % % % % % % % % % % % % % % % % % % % % % % % % % % % % % % % % % % % % % % % % % % % % % % % % % %
%                             A P P E N D I X :

\section{Appendix B. Naturally parametrized geodesics}

Naturally parametrized geodesics can be defined as extremals of the action functional
\begin{equation*}
J(\gamma) = \int\limits_{\gamma} \bigl(a{\dot x}^2 + 2b{\dot x}\dot y + c{\dot y}^2\bigr)\,dt,
\quad
\dot x = \frac{dx}{dt}, \ \ \dot y = \frac{dy}{dt},
\end{equation*}
where $\gamma \subset S$ is a differentiable curve.
The corresponding Euler-Lagrange equation reads
\begin{equation}
\left \{ \
\begin{aligned}
& 2(a \ddot x + b \ddot y) = (c_x-2b_y) {\dot y}^2 - 2a_y {\dot x} {\dot y} - a_x {\dot x}^2, \\
& 2(b \ddot x + c \ddot y) = (a_y-2b_x) {\dot x}^2 - 2c_x {\dot x} {\dot y} - c_y {\dot y}^2. \\
\end{aligned}
\right.
\label{ELE}
\end{equation}
The definition of geodesics as auto-parallel curves in the Levi-Civita connection
generated by the metric \eqref{1} leads to the same equation~\eqref{ELE}.
Equation~\eqref{ELE} defines a direction field on the tangent bundle $TS$.
The standard projectivization $TS \to PTS$ sends this direction field to the field parallel to \eqref{5},
see \cite{Rem15}.

\medskip

Firstly, using equation~\eqref{ELE} of parametrized geodesics, we prove the omitted statement
in the case $Z$ that the line $\Psi (\Pi) = \{y=p=0\}$
does not correspond to a geodesic.
Recall that in the case $Z$ there exist local coordinates such that
$$
ds^2 = (y\om + \ldots) dx^2 + (0 + \ldots) dxdy - (\om + \ldots)dy^2,
\quad \om(0,0)=-1.
$$
where the dots mean terms that belong to the ideal $\MM^{\infty}(y)$.

Using appropriate change of variables $y \mapsto yu(x,y)$, where $u$ is a solution of equation
$cu_x + b/y = 0$ with the condition $u(0,0) \neq 0$,
one can bring locally the metric to the diagonal form $ds^2 = a dx^2 + c dy^2$ with
the coefficients $a = yu\om + \ldots$ and $c = -\om (u+yu_y)$.
Substituting $y \equiv b \equiv 0$ into the second equation of \eqref{ELE}, we get
$a_y(x,0) {\dot x}^2 = 0$. Since $a_y(0,0) \neq 0$, this yields $\dot x \equiv 0$,
and the restriction of the system \eqref{ELE} to $y=0$ has only constant solutions,
which are not geodesics.

At first sight it contradicts the fact established in \cite{GR}:
$\FF$ is an invariant surface of the field \eqref{5}, and consequently,
any trajectory of \eqref{5} that lies entirely in $\FF$
after the projection on the $(x,y)$-plane gives a geodesic or a point.
(Example: for the metric $ydx^2 - dy^2$ the isotropic surface $p^2=y$ is filled with
one-parameter family of integral curves intersecting $\Psi (\Pi)$ transversally.
Projecting this family down, we get the isotropic geodesics $y = \frac{1}{4}(x-c)^2$.)
In fact, there is no contradiction: the curve $\Psi (\Pi) \subset \FF$ consists of
singular points of the field \eqref{5}, and every such a point is a trajectory of \eqref{5}.

\medskip

Consider the family $\Gamo$ of geodesics outgoing from a point $q \in \D$
with the isotropic direction $p_0$.
Choose the natural parametrization so that the motion along
geodesics proceeds toward $q$.
In the paper \cite{Rem15}, it was proved that in the cases $C_1, C_3$
any geodesic $\gam \in \Gamo$ incomes into the point $q$ in finite time with infinite velocity.
The same statement is valid in the case $C_2$, since it deals with the root $p_0$ only.

In the case $Z$, the same result follows from the asymptotic formula established in Theorem~\ref{TAPP} below.
The cases $D_s, D_n$ can be considered similarly, the case $D_f$ is excluded from consideration,
since every geodesic $\gam \in \Gamo$ intersects the discriminant curve $\D$ infinite number of times
in any neighborhood of $q$.

\begin{theorem}
\label{TAPP}
The natural parametrization of geodesics \eqref{17} is given by the formula
$x = t^{\frac{1}{3}} \bigl(1+X_{\alf}(t^{\frac{1}{3}})\bigr)$,
where $X_{\alf}(\cdot)$ are smooth functions vanishing at zero.
\end{theorem}

\begin{proof}
Choosing the local coordinates in Theorem~\ref{T3},
from the formula \eqref{17} we have $y=\frac{1}{4}x^2+O(x^4)$, $\dot y=(\frac{1}{2}x+O(x^3))\dot x$,
and $\ddot y=(\frac{1}{2}x+O(x^3))\ddot x + (\frac{1}{2}+O(x^2)){\dot x}^2$.
Substituting these expressions together with the coefficients $a,b,c$ from \eqref{11}
into the first equation in \eqref{ELE},
after a straightforward transformation we obtain
\begin{equation}
\frac{\ddot x}{\dot x} = \Bigl( -\frac{2}{x} + f_{\alf}(x) \Bigr) \dot x,
\label{5.2}
\end{equation}
where $f_{\alf}(x)$ are smooth functions.

Equation \eqref{5.2} defines the natural parametrization uniquely up to
non-degenerate affine transformations of the $t$-axis.
Integrating it, we get $\ln |\dot x| = -2\ln |x| + F_{\alf}(x) + C$,
and $\dot x = K x^{-2} e^{F_{\alf}(x)}$, where $F_{\alf}$ is the primitive of $f_{\alf}$.
Without loss of generality put $F_{\alf}(0)=0$ and $K = \frac{1}{3}$
(this corresponds to the choice of the initial velocity of motion along the geodesic).
Then we arrive at the differential equation $\frac{dt}{dx} = 3x^{2} e^{- F_{\alf}(x)}$,
whose general solution is $t = x^3 (1+T_{\alf}(x)) + t_0$,
where $T_{\alf}(x)$ is a smooth function vanishing at zero.
Setting $t_0=0$ and inverting, we get
$x = t^{\frac{1}{3}} \bigl(1+X_{\alf}(t^{\frac{1}{3}})\bigr)$.
\end{proof}

As an example, for the metric $ds^2 =  dy^2 - y dx^2$ the system~\eqref{ELE} reads
$y \ddot x = - \dot x \dot y$, $2 \ddot y = - {\dot x}^2$.
Substituting here the isotropic geodesic $y = \frac{1}{4} x^2$, we get
$x = k(t-t_0)^{\frac{1}{3}}$.
Substituting  $y = 0$ (the line $\Psi (\Pi)$), we get $\dot x = 0$, that is,
$y=0$ is not a geodesic as was stated above.

%%%%%%%%%%%%%%%%%%%%%%%%%%%%%%%%%%%%%%%%%%%%%%%%%%%%%%%%%%%%%%%%%%%%

\noindent
Instituto de Ci\^encias Matem\'aticas e de Computa\c{c}\~ao - USP, Avenida Trabalhador S\~ao-Carlense, 400 - Centro,
CEP: 13566-590 - S\~ao Carlos - SP, Brazil.\\
Emails: alexey-remizov@yandex.ru, faridtari@icmc.usp.br

\end{document}